\documentclass[12pt]{amsart}
\usepackage{amsmath}
\usepackage{epsfig}
\usepackage{graphics,graphicx}
\usepackage{amsfonts}
\usepackage{caption}
\usepackage{verbatim}
\usepackage{amssymb}
\usepackage{latexsym}
\usepackage{amsmath, amsthm, amssymb}
\usepackage[usenames]{color}
\usepackage{hyperref}
\usepackage[pagewise]{lineno}

\newtheorem{theo}{Theorem}[section]

\newtheorem{prop}[theo]{Proposition}

\newtheorem{defi}[theo]{Definition}

\newtheorem{properties}[theo]{Properties}

\newcommand{\R}{\mathbb{R}}
\numberwithin{equation}{section}
\def\w{\textit {\textbf w}}
\def\z{\textit {\textbf z}}
\def\u{\textit {\textbf u}}
\def\n{\textit {\textbf n}}

\newcommand{\N}{\mathbb{N}}
\newcommand{\Z}{\mathbb{Z}}

\newcommand{\mylabel}[1]{\label{#1}
            \ifx\undefined\stillediting
            \else \fbox{$#1$}\fi }
\newcommand{\BE}{\begin{equation}}

\newcommand{\EEQ}{\end{equation}}
\newcommand{\rfb}[1]{\mbox{\rm
   (\ref{#1})}\ifx\undefined\stillediting\else:\fbox{$#1$}\fi}

\newfont{\Blackboard}{msbm10 scaled 1200}

\newfont{\roma}{cmr10 scaled 1200}

\def\CC{\rm \hbox{C\kern-.56em\raise.4ex
         \hbox{$\scriptscriptstyle |$}\kern+0.5 em }}

\newcommand{\mm}    {{\hbox{\hskip 0.5pt}}}

\newcommand{\bluff} {{\hbox{\raise 15pt \hbox{\mm}}}}
\makeatletter
\def\section{\@startsection {section}{1}{\z@}{-3.5ex plus -1ex minus
    -.2ex}{2.3ex plus .2ex}{\large\bf}}
\makeatother
\def\be{\begin{equation}}
\def\ee{\end{equation}}

\def \Z {{\mathbb{Z}}}
\def \R {{\mathbb{R}}}
\def \N {{\mathbb{N}}}

\begin{document}

\thispagestyle{empty}
\title[]{$L^p$-Strong solution for the stationary exterior Stokes equations with Navier boundary condition}
\date\today

\author{Dhifaoui Anis}
\address{UR Analysis and Control of PDEs, UR 13ES64, Department of Mathematics, Faculty of Sciences of Monastir, University of Monastir, Tunisia }
\email{anisdhifaoui123@gmail.com}

\begin{abstract}
This paper treats the stationary Stokes problem in exterior domain of $\R^3$ with Navier slip boundary condition. The behavior at infinity of the data and the solution are determined by setting the problem in $L^p$-spaces, for $p> 2$, with weights. The main results are the existence and uniqueness of strong solutions of the corresponding system.

\end{abstract}


\maketitle

\tableofcontents

 
\section{Introduction} \label{secintro}

We consider a simply-connected bounded domain $\Omega' \subset\R^3$, of class $\mathcal{C}^{2.1}$ with boundary $\Gamma$. Let $\Omega$ be an exterior domain given by $\Omega:=\R^3 \setminus \overline{\Omega'}$. The motion of viscous incompressible fluid in the exterior domain $\Omega$ around the obstacle $\Omega'$ is described by the Navier-Stokes equations, which are non-linear. The Stokes systems is a linear approximation of this model. Precisely, the velocity $\u$ of fluid and the pressure $\pi$ satisfy the following stationary Stokes problem: 
\begin{equation}
\label{ST}
\begin{split}
-\nu\Delta\textit{\textbf{u}}+\nabla\pi&=\textit{\textbf{f}}\quad\text{in}\quad\Omega,\\
\mathrm{div }\,\textit{\textbf{u}}&=0\quad\text{in}\quad\Omega,
\end{split}
\end{equation} 
where  $\textit{\textbf{f }}$ the external forces acting on the fluid, there are several possibilities of boundary conditions. Under the hypothesis of impermeability of the boundary, the velocity field $\u$ satisfies:
\begin{equation}\label{normal}
\begin{split}
\u\cdot\n=0\quad\text{in}\quad\Gamma,
\end{split}
\end{equation}
where $\n$ stands for the outer normal vector. According to the idea that the fluid cannot slip on the wall due to its viscosity, we get the no-slip condition:
\begin{equation}\label{tangential}
\begin{split}
\u_{\tau}=\mathbf{0}\quad\text{in}\quad\Gamma,
\end{split}
\end{equation}
where $\u_{\tau}=\u -(\u\cdot\n)\n$ denotes, as usual, the tangential component of $\u$. The problem~\eqref{ST} in an exterior domain with the Dirichlet boundary condition, is the combination of~\eqref{normal}--\eqref{tangential}, has been studied by many authors. We can mention~\cite{Alliot_M3AS_2000, Vivet_1994,Girault_M2AS_1992, Girault_ARMA_1991, 
Specovius_M2AS_1986, Specovius_AAM_1994} and references therein. However, there are many other kinds of boundary conditions which also match in the reality. In the physical applications, we are often encountering situations where the no-slip boundary conditions does not quite feasible. In this case, it is really important to introduce another boundary condition to describe the behavior of fluid on the wall. For example, hurricanes and tornadoes do slip along the ground and lose energy as they slip~\cite{Oliveira-07}. For the skin of sharks (\cite{Friedmann_JMFM_2010, Friedmann_JMFM_2011}) or golf balls in the case that the obstacle may have rough boundaries. Another application of interest can also be found, for instance, in aerodynamics in drag control of aircraft wings, in order to reduce the drag, small injection jets are introduced over the wings of the plane ~\cite{Achdou_collectionJLL_1998}. In 1827,  C. Navier~\cite{NAVIER}  was the first mathematician who considered the slip phenomena and proposed the following boundary conditions, called Navier-slip boundary conditions:
\begin{eqnarray}\label{Navier.BC}
\left\{\begin{array}{l}
\mathbf{u} \cdot \mathbf{n}=0 \\
2 [\mathbf{D}(\mathbf{u}) \mathbf{n}]_\tau +\alpha(\mathrm{x}) \mathbf{u}_\tau=0
\end{array} \quad \text { on } \quad \Gamma,\right.
\end{eqnarray}
\noindent where $\mathbf{D}(\u)=\frac{1}{2}\left(\nabla \mathbf{u}+\nabla^{T} \mathbf{u}\right)$, $\mathbf{n}$ and $\tau$ are unit outer normal vector and tangential vector of the boundary $\Gamma$ . In~\eqref{Navier.BC}, $\alpha(x)$ is a physical parameter, which can be a positive constant or a function in $L^{\infty}(\Gamma)$. Here, we consider the case that $\alpha(x)$ is constant which called the slip coefficient. The first condition in~\eqref{Navier.BC} is the no-penetration condition and the second condition expresses the fact that the tangential 
velocity, instead of being zero as in the slip condition~\eqref{tangential}, is proportional to the tangential stress. The Navier slip conditions have been extensively studied, see for instance~\cite{Achdou_CRAS_1995, Achdou_JCP_1998, Basson_CPAM_2008, 
Casado-Diaz_JDE_2003, Gerard_Varet_CMP_2010, Jager_JDE_2001, Beavers_JFM_1967, Solonnikov_TMIS_1973} and references therein. For the case of Navier boundary conditions without friction ($\alpha=0$), let us mention~\cite{Meslamani_2013, LMR_2020, Anis21}, where the following boundary conditions were used:
\begin{eqnarray}\label{Navier.flat.BC}
\left\{\begin{array}{l}
\mathbf{u} \cdot \mathbf{n}=0 \\
 \mathbf{curl}\,\mathbf{u}\times\mathbf{n}=0
\end{array} \quad \text { on } \quad \Gamma,\right.
\end{eqnarray}
where $\mathbf{curl }\textit{\textbf{u}}$ is the vorticity field. These conditions coincide with~\eqref{Navier.BC} on flat boundaries when $\alpha=0$.
They were also used in~\cite{Beirao_CPAA_2006} for the study of the non stationary Navier-Stokes equations in half-spaces of $\R^3$.
We finally refer to~\cite{Mulone_Meccanica_1983, Mulone_AnnMatPuraAppl_1985}
for the study of the non stationary problem of Navier-Stokes with mixed boundary conditions that include~\eqref{Navier.BC} without friction.\\

\noindent The problem~\eqref{ST}--~\eqref{Navier.BC} set in bounded domains has been well studied by various authors (see for instance~\cite{Ahmed_2014, Beirao_ADE_2004} or~\cite{Amrouche_M2AS_2016, Solonnikov_TMIS_1973} for the case $\alpha=0$
and references therein). Although in the exterior domain,  to the best of our knowledge, we can just mention~\cite{Russo_JDE_2011} where~\eqref{Navier.BC} was used for the stationary Navier-Stokes equations in exterior domains with also the assumption that the velocity tends to a non zero constant vector at infinity. But in the case that the velocity tends to a zero at infinity we can mention~\cite{DMR-2019, dhifaoui2021very}, where the authors studied the Stationary Stokes problem with Navier slip boundary condition in an exterior domain, they posed the problem in weighted spaces in order to control the infinite behavior of the solutions. They obtained in the Hilbertian framework, existence results, uniqueness of variational and strong solutions and another class of solutions called very weak solution for less regular data. \\

Since the domain $\Omega$ is unbounded, we set the problem in weighted Sobolev spaces, the weight functions are polynomials and enable to describe the growth or the decay of functions at infinity which allows to look for solutions of the problem with various behaviors at infinity this is one of the main advantages of the weighted Sobolev spaces. In this work, we study the problem~\eqref{ST}--\eqref{Navier.BC} in $L^p$-theory where $p>2$, we look for the strong solutions that have a different behavior at infinity. To that end, we combine results on the Stokes problem 
set in the whole space to catch the behavior at infinity and results in bounded domains to take into account the boundary conditions. \\

The paper is organized as follows. In Section~\ref{section 2}, we introduce the notations, the functional framework based on weighted Sobolev space. We recall some basic results concerning the Stokes in the whole space, we give a result concerning Laplace problem with Neumann boundary conditions. We end this
section by solving a mixed Stokes problem with Navier and Dirichlet boundary conditions. Finally, in Section~\ref{sect.string.solution}, we prove the existence and uniqueness of Strong solutions for the exterior Stokes problem~\eqref{ST}--\eqref{Navier.BC}. 
%

\section{Notations and Preliminary Results}\label{section 2}
\subsection{Notations}\quad\\
In this section, we recall the main notation which we shall use later. In what follows, $p$ is a real number in the interval $]1,\infty[$. The dual exponent of $p$ denoted by $p'$ is given by the relation $\frac{1}{p}+\frac{1}{p'}=1$. We will use bold characters for vector and matrix fields. A point in $\R^3$ is denoted by $\textbf{\textit{x}}=(x_1,x_2,x_3)$ and its
distance to the origin by
$$r=|\textbf{\textit{x}}|=\left(x_1^2+x_2^2+x_3^2\right)^{1/2}.$$
 For any multi-index $\boldsymbol{\lambda}\in\N^3$, we denote
by $\partial^{\boldsymbol{\lambda}}$ the differential operator of order
$\boldsymbol{\lambda}$,
$$D^{\boldsymbol{\lambda}}=
{\partial^{|\lambda|}\over\partial_1^{\lambda_1}\partial_2^{\lambda_2}\partial_3^{\lambda_3}},
\quad|\lambda|=\lambda_1+\lambda_2+\lambda_3.$$
\noindent We denote by $[s]$ the integer part of $s$. For any $k\in\Z$, $\mathcal{P}_k$ stands for the space of polynomials of degree less than or equal to $k$ and $\mathcal{P}^{\Delta}_{k}$ the harmonic polynomials of $\mathcal{P}_k$. If $k$ is a negative integer, we
set by convention $\mathcal{P}_k=\{0\}$. Let $\Omega'\subset\R^3$ is a simply connected bounded domain that has a boundary $\partial \Omega'=\Gamma$ of class $\mathcal{C}^{2.1}$ and let $\Omega $ be the complement of its closure in  $\R^{3}$. We denote by $\mathcal{D}(\Omega)$ the space of
$\mathcal{C}^{\infty}$ functions with compact support in $\Omega$, $\mathcal{D}(\overline{\Omega})$ the restriction to $\Omega$ of functions belonging to 
$\mathcal{D}(\R^3)$. 
We recall that $\mathcal{D}'(\Omega)$ is the well-known space of distributions defined on
$\Omega$. We recall that $L^p(\Omega)$ is the well-known Lebesgue real space and for $m\ge1$, we recall that $W^{m,p}(\Omega)$ is the well-known classical Sobolev space. 
We shall write $u\in
W_{loc}^{m,p}(\Omega)$ to mean that $u\in W^{m,p}(\mathcal{O})$, for any
bounded domain $\mathcal{O}$, with $\overline{\mathcal{O}}\subset\Omega$. In this work, we shall also denote by $B_{R}$ the open ball of radius $R>0$ centred at the origin with boundary $\partial B_R$. In particular, since $\Omega'$ is bounded, we can find some $R_{0}$, such that $\Omega'\subset B_{R_{0}}$ and we introduce, for any $R\geq R_{0}$, the set $\Omega_{R}=\Omega \cap B_{R}.$ If $X$ is a Banach space, with dual space $X'$, and $Y$ is a closed subspace of $X$, we denote by $X'\perp Y$ the subspace of $X'$ orthogonal to $Y$, i.e.
$$
X'\perp X
Y=\lbrace f \in X', \,\, \forall\, v \in Y,\,\, <f,v>=0\rbrace=(X/Y)'.
$$

\noindent Given $\textit{\textbf{A}}$ and $\textit{\textbf{B}}$ two matrices fields, such that $\textit{\textbf{A}}=(a_{ij})_{1\leqslant i,j\leqslant 3}$ 
and $\textit{\textbf{B}}=(b_{ij})_{1\leqslant i,j\leqslant 3}$, then we define 
$\textit{\textbf{A}}:\textit{\textbf{B}}=(a_{ij}b_{ij})_{1\leqslant i,j\leqslant 3}$.
Finally, as usual, $C>0$ denotes a generic constant the value
of which may change from line to line and even at the same line.\\

\subsection{Weighted Sobolev spaces}\quad\\
\noindent In order to control the behavior at infinity of our functions and distributions we use for basic weights the quantity  $\rho(\textbf{\textit{x}})=(1+r^2)^{1/2}$ which is equivalent to $r$ at infinity, and to one on any bounded subset of $\R^3$.\\
For $k\in\Z$, we introduce
$$W_k^{0,p}(\Omega)=\Big\{u\in\mathcal{D}'(\Omega),\,\rho^k u\in L^p(\Omega)\Big\},$$
which is a Banach space equipped with the norm: 
$$\|u\|_{W_k^{0,p}(\Omega)}=\|\rho^k u\|_{L^p(\Omega)}.$$ 
\noindent For any non-negative integers $m$, real numbers $p>1$ and $k\in \Z$. We define the weighted Sobolev space for $3/p+k\notin \{1,\cdots,m\}$:
$$
W_{k}^{m,p}(\Omega)=\Big\{u\in \mathcal{D}'(\Omega);\,\forall\boldsymbol{\lambda}\in\N^{3}:
\,0\leq |\boldsymbol{\lambda}| \leq m,\,\rho^{k-m+|\boldsymbol{\lambda}|}D^{\boldsymbol{\lambda}}u \in L^{p}(\Omega) \Big\}.
$$
It is a reflexive Banach space equipped with the norm:
$$
\|u\|_{W_{k}^{m,p}(\Omega)}= \left(\sum_{0\leqslant|\boldsymbol{\lambda}|\leqslant m}
\|\rho^{k-m+|\boldsymbol{\lambda}|}D^{\boldsymbol{\lambda}}u\|^{p}_{L^{p}(\Omega)}\right)^{1/p}.
$$
\noindent We define the semi-norm
$$|u|_{W_k^{m,p}(\Omega)}=\left(\sum_{|\boldsymbol{\lambda}|=m}\|\rho^k D^{\boldsymbol{\lambda}}u\|_{L^p(\Omega)}\right)^{1/p}.$$

\noindent Let us give some examples of such space that will be often used in the remaining of this work.
\begin{itemize}
\item[1)] For $m=1$, we have
\begin{align*}
{W_{k}^{1,p}(\Omega)}:={\{}&u \in \mathcal{D}'(\Omega);\,\rho^{k-1}u \in L^{p}(\Omega),\,\rho^{k}\,\nabla\,u \in L^{p}(\Omega) \}
\end{align*}

%
\item[2)] For $m=2$, we have
\begin{small}

\begin{equation*}
W_{k+1}^{2,p}(\Omega):=\left\lbrace u\in {W_{k}^{1,p}(\Omega)}, \rho^{k+1}\nabla^{2} u \in L^{p}(\Omega)\right\rbrace, 
\end{equation*}
\end{small}
\end{itemize}

Now, we present some basic properties on weighted Sobolev spaces. For more details, the reader can refer to  \cite{Amrouche_JMPA_1997, Amrouche_1994, Hanouzet}. 
\begin{properties}\quad
\begin{itemize}
\item The space $\mathcal{D}(\overline{\Omega})$ is dense in $W^{m,p}_k(\Omega)$.\\
\item For any $m\in\N^*$ and $3/p+k\neq 1$, we have the following continuous embedding:
\begin{equation}\label{inclusion.sobolev2}
W_{k}^{m,p}(\Omega)\hookrightarrow W_{k-1}^{m-1,p}(\Omega).
\end{equation}
\item For any $k$, $m\in\Z$ and for any $\lambda \in \N^3$, the mapping
\begin{equation}\label{derive.espaces.poids}
u \in W_{k}^{m,p}(\Omega)\longrightarrow\,\,\partial^{\lambda}u\in W_{k}^{m-|\lambda|,p}(\Omega)
\end{equation}
is continuous.\\
\item If $3/p+k\notin \{1,\cdots,m\}$, $3/p+k-\mu\notin \{1,\cdots,m\}$ and $m\in\Z$ the mapping
$$u\in W^{m,p}_{k}(\Omega)\longrightarrow \rho^{\mu}u\in W^{m,p}_{k-\mu}(\Omega)$$
is an isomorphism.
\end{itemize}
\end{properties}
The space $W_{k}^{m,p}(\Omega)$  sometimes contains some polynomial functions. Let $j$ be defined as follow:
\begin{equation}
j=\begin{cases}
[m-(3/p+k)] \quad\quad\mathrm{if}\quad 3/p+k \notin \Z^{-},\\
m-3/p-k-1\qquad\qquad \mathrm{otherwise}.
\end{cases}
 \end{equation}
 Then $\mathcal{P}_j$ is the space of all polynomials included in $W_{k}^{m,p}(\Omega)$.

\noindent The norm of the quotient space $W_{k}^{m,p}(\Omega)/\mathcal{P}_{j}$ is given by:
\begin{eqnarray*}
||u||_{W_{k}^{m,p}(\Omega)/\mathcal{P}_{j}}=\inf_{\mu\in\mathcal{P}_{j}}||u+\mu ||_{W_{k}^{m,p}(\Omega)}.
\end{eqnarray*}
All the local properties of $W_{k}^{m,p}(\Omega)$ coincide with those of the corresponding classical Sobolev spaces $W^{m,p}(\Omega)$. Hence, it also satisfies the usual trace theorems on the boundary $\Gamma$. Therefore, we can define the space

$$\mathring{W}_{k}^{m,p}(\Omega)=\lbrace u\in  W_{k}^{m,p}(\Omega),\,\,\, \gamma_{0}u=0,\,\gamma_{1}u=0,\,\cdots, \gamma_{m-1}u=0\rbrace.$$

\noindent Note that when $\Omega=\R^{3}$, we have $\mathring{W}_{k}^{m,p}(\R^{3})=W_{k}^{m,p}(\R^{3})$. The space $\mathcal{D}(\Omega)$ is dense in $\mathring{W}^{m,p}_k(\Omega)$. Therefore, the dual space of $\mathring{W}_{k}^{m,p}(\Omega)$, denoting  by $W_{-k}^{-m,p'}(\Omega)$, is a space of distributions with the norm
\begin{equation*}
||\,u\,||_{W_{-k}^{-m,p'}(\Omega)}=\sup_{v \in \mathring{W}_{k}^{m,p}(\Omega)}\dfrac{\left\langle u,v\right\rangle_{W_{-k}^{-m,p'}(\Omega)\times \mathring{W}_{k}^{m,p}(\Omega)} }{||\,v\,||_{W_{k}^{m,p}(\Omega)}}.
\end{equation*}

\noindent The proof of the following theorem can be found in \cite[Proposition 2.1]{Louati_Meslameni_Razafison}.
\begin{theo}\quad\\
Let $k$, $l$ be real numbers. Let $\lambda$ be a polynomial that belongs to $W^{1,p}_{k}(\Omega)+ W^{1,q}_{l}(\Omega)$. Then $\lambda$ belongs to $\mathcal{P}_{\gamma}$ where
\begin{equation*}
\gamma=\max \big([1-\frac{3}{p}-k],[1-\frac{3}{q}-l]\big).
\end{equation*}
\end{theo}


  
 
\noindent  We note that the vector-valued Laplace operator of a vector field $\textbf{\textit{v}}=(v_{1},\,v_{2},\,v_{3})$ is equivalently defined by
\begin{equation}\label{delta2}
\Delta\,\textbf{\textit{v}}=2\mathrm{div}\mathrm{\textbf{D}}(\textbf{\textit{v}})-\mathbf{grad}(\textbf{\textit{v}})
\end{equation}

\vspace{.2cm}
\noindent This leads to the following definition.
\begin{defi}{\rm 
For all integers $k \in \Z$ and $1<p<\infty$.
The space  ${H}_{k}^{p}(\mathrm{div},\Omega)$ is defined by
\begin{equation*}
 {H}_{k}^{p}(\mathrm{div},\Omega)=\left\lbrace \textbf{\textit{v}}\in {W}_{k}^{0,p}(\Omega);\mathrm{div} \,\, \textbf{\textit{v}}\in W_{k+1}^{0,p}(\Omega)\right\rbrace \,,
\end{equation*}
and is provided with the norm
$$ 
\|\textbf{\textit{v}}\|_{{H}_{k}^{p}(\mathrm{div},\Omega)}=\left( \|\textbf{\textit{v}}\|^{p}_{{W}_{k}^{0,p}(\Omega)}+\|\mathrm{div} \,\, \textbf{\textit{v}}\|^{p}_{{W}_{k+1}^{0,p}(\Omega)}\right) ^{\frac{1}{p}}.
 $$
\vspace{.2cm}
\noindent
This definition will be also used with $\Omega$ replaced by ${\R}^{3}$.
\vspace{.3cm}
}
\end{defi}
\noindent The argument used by Hanouzet (see \cite{Hanouzet_1971}) to prove the density of $\mathcal{D}(\overline{\Omega})$ in $W_{k}^{m,p}(\Omega)$ can be adapted to establish that $\boldsymbol{\mathcal{D}}(\overline{\Omega})$ is dense in ${H}_{k}^{p}(\mathrm{div},\Omega)$. Therefore, denoting by $\textbf{\textit{n}}$ the exterior unit normal to the boundary $\Gamma$, the normal trace $\textbf{\textit{v}}\cdot \textbf{\textit{n}}$ can be defined in $W^{-1/p,p}(\Gamma)$ for the functions of ${H}_{k}^{p}(\mathrm{div},\Omega)$, where $W^{-1/p,p}(\Gamma)$ denotes the dual space of $W^{1/p,p'}(\Gamma)$. They satisfy the trace theorems; i.e, there exists a constant $C$ such that 
\begin{eqnarray}
\forall\textbf{\textit{v}}\in {H}_{k}^{p}(\mathrm{div},\Omega),\quad \Vert\textbf{\textit{v}}\cdot\textbf{\textit{n}}\Vert_{{\textit{W}}^{-1/p,p}(\Gamma)}\leqslant C\Vert\textbf{\textit{v}}\Vert_{\textbf{\textit{H}}_{k}^{p}(\mathrm{div},\Omega)}.
\end{eqnarray}
In addition, the following Green's formulas holds: For any $\textbf{\textit{v}}\in {H}_{k}^{p}(\mathrm{div},\Omega)$ and ${\varphi}\in {{W}} _{-k}^{1,p'}(\Omega)$
\begin{equation}\label{FG1}
 \displaystyle \int_{\Omega} \textbf{\textit{v}} \cdot \nabla {\varphi} \,d\textbf{x}+ \displaystyle \int_{\Omega} {\varphi}\,\mathrm{div}\,  \textbf{\textit{v}}\, d\textbf{x}=\left\langle \textbf{\textit{v}} \cdot \textbf{\textit{n}},\varphi \right\rangle _{{W}^{-1/p,p}(\Gamma) \times {W}^{1/p,p'}(\Gamma)}.
\end{equation}
\vspace{.3cm}
The closures of $\boldsymbol{\mathcal{D}}({\Omega})$ in ${H}_{k}^{p}(\mathrm{div},\Omega)$ is denoted by $\mathring{H}_{k}^{p}(\mathrm{div},\Omega)$  and can be characterized by
\begin{eqnarray*}
\mathring{H}_{k}^{p}(\mathrm{div},\Omega)=\lbrace \textbf{\textit{v}}\in {H}_{k}^{p}(\mathrm{div},\Omega): \textbf{\textit{v}}\cdot\textbf{\textit{n}}=0\,\text{ on }\, \Gamma\rbrace.
\end{eqnarray*}
The proof of the following result can be found in~\cite{dhifa2020}:
\begin{prop}\label{prop11}
A distribution  $\textbf{\textit{f}}$ belongs to $[\mathring{{H}}_{k}^{p}(\mathrm{div},\Omega)]'$ if and only if there exist $\boldsymbol{\psi}\in {W}^{0,p}_{-k}(\Omega)$ and $\boldsymbol{\chi}\in {W}^{0,p}_{-k-1}(\Omega)$, such that $\textbf{\textit{f}}=\psi+\nabla\chi$. Moreover
\begin{eqnarray*}
\Vert \boldsymbol{\psi}\Vert_{{W}^{0,p}_{-k}(\Omega)} + \Vert \boldsymbol{\chi}\Vert_{{W}^{0,p}_{-k-1}(\Omega)}\leq C \Vert \textbf{\textit{f}}\Vert_{[\mathring{H}_{k}^{p}(\mathrm{div},\Omega)]'}.
\end{eqnarray*}
\end{prop}
\noindent 

\subsection{The Stokes problem in the whole space $\R^3$}\quad\\
\noindent We recall here some basic results concerning the Stokes problem in $\R^{3}$: 

  \begin{equation}\label{Stokes in R^3}
  -\Delta\,\textbf{\textit{u}}+\nabla\,{\pi}=\textbf{\textit{f}}\quad\text{and}\quad
\mathrm{div}\,\textbf{\textit{u}}={\chi}\quad\text{in}\quad\R^3\, .
  \end{equation}
  These results can be found in~\cite{Alliot_M3AS_1999}. Let us first introduce the kernel of the Stokes operator
  \begin{eqnarray*}
\mathcal{N}_{k}^{p}(\R^3)=\Big\lbrace(\textbf{\textit{u}},{\pi})\in W^{1,p}_{k}(\R^3)\times{{W}}^{0,p}_{k}(\R^3),\,\,\,-\Delta\,\textbf{\textit{u}}+\nabla\,{\pi}=\boldsymbol{0}\,\,\text{ and }\,\, \mathrm{div}\,\textbf{\textit{u}}=0\,\,\text{ in }\,\R^3\Big\rbrace 
\end{eqnarray*}
and the space of polynomials
\begin{eqnarray*}
N_{k}=\Big\lbrace(\boldsymbol{\lambda},\mu)\in \mathcal{P}_{k}\times\mathcal{P}^{\Delta}_{k-1},\,\,\,-\Delta\,\boldsymbol{\lambda}+\nabla\,\mu
=\boldsymbol{0}\,\,\text{ and }\,\, \mathrm{div}\,\boldsymbol{\lambda}=0\Big\rbrace .
\end{eqnarray*}
Recall that by agreement on the notation $\mathcal{P}_{k}$, the space $N_{k}=\lbrace (\boldsymbol{0},0)\rbrace$ when $k<0$ and $N_{0}=\mathcal{P}_{0}\times\lbrace 0\rbrace$.\\

 The next proposition characterizes the kernel of~\eqref{Stokes in R^3}.
\begin{prop}\label{caracterisation du noyau dans R^3}
Let $1<p<\infty$ and k be integers. Then $\mathcal{N}_{k}^{p}(\R^3)=N_{[1-k-3/p]}$. In particular, $\mathcal{N}_{k}^{p}(\R^3)=\lbrace(\mathbf{0},0)\rbrace$ if $k>1-3/p$.
\end{prop}

\noindent The next theorem states an existence, uniqueness and regularity result for problem~\eqref{Stokes in R^3}.

\begin{theo}
\label{theorem de Stokes dans R^3}
Assume that $k\in\Z$ and $1<p<\infty$. If $(\textbf{\textit{f}},\chi)\in W^{-1,p}_{k}(\R^3)\times W^{0,p}_{k}(\R^3)$ satisfies the compatibility condition:
\begin{equation}
\label{CCR^3}
\forall (\boldsymbol{\lambda},\mu)\in N_{[1-3/p'+k]},\quad \langle\textbf{\textit{f}},\boldsymbol{\lambda}\rangle_{W^{-1,p}_{k}(\R^3)\times W^{1,p'}_{-k}(\R^3)} 
-\langle \chi,\mu\rangle_{W^{0,p}_{k}(\R^3)\times W^{0,p'}_{-k}(\R^3)}=0,
\end{equation}
then problem \eqref{Stokes in R^3} has a solution $(\textbf{\textit{u}}, \pi)\in W^{1,p}_{k}(\R^3)\times W^{0,p}_{k}(\R^3)$ unique up to an element of $N_{[1-3/p-k]}$ and we have the estimate
\begin{eqnarray*}\label{estimation du problem in R^3}
\inf _{(\boldsymbol{\lambda},{\mu})\in N_{[1-3/p-k]}}\left( \Vert \textbf{\textit{u}}+\boldsymbol{\lambda}\Vert_{W^{1,p}_{k}(\R^3)}+\Vert {\pi}+{\mu}\Vert_{{{W}}^{0,p}_{k}(\R^3)}\right) \leqslant C\Big(\Vert \textbf{\textit{f }}\Vert_{W^{-1,p}_{k}(\R^3)}+\Vert \chi\Vert_{W^{0,p}_{k}(\R^3)}\Big).
\end{eqnarray*}
Furthermore, if $(\textbf{\textit{f}},\chi)\in W^{0,p}_{k+1}(\R^3)\times W^{1,p}_{k+1}(\R^3)$, then $(\textbf{\textit{u}}, \pi)\in W^{2,p}_{k+1}(\R^3)\times W^{1,p}_{k+1}(\R^3)$.
\end{theo}
\noindent The proofs of Proposition~\ref{caracterisation du noyau dans R^3} and Theorem~\ref{theorem de Stokes dans R^3} can be found in~\cite{Alliot_M3AS_1999}.

\subsection{Generalized Neumann problem}\quad\\
In this section, we are interested into the following Neumann problem:
\begin{eqnarray}\label{Problem de Neumann Ax}
-\Delta\,\textbf{\textit{u}}=\textbf{\textit{f}}\quad\text{ in}\,\,\,\Omega\quad\text{ and }\quad
\dfrac{\partial\textbf{\textit{u}}}{\partial\textbf{\textit{n}}}=g\quad\text{ on}\,\,\,\Gamma.
\end{eqnarray}
Our first proposition is established also in \cite{Louati_Meslameni_Razafison}, it characterizes the kernel of the Laplace operator with Neumann boundary condition. For any integer $k\in \Z$ and $1<p<\infty$,
\begin{equation*}
\mathcal{N}_{p,k}^{\Delta}=\left\lbrace \w\in W_{k}^{1,p}(\Omega);\,\,\,\,\Delta\,\w=0\quad\mathrm{in}\,\, \Omega\quad\mathrm{and}\quad \dfrac{\partial \w}{\partial \textbf{\textit{n}}}= 0\quad\mathrm{on}\,\, \Gamma\right\rbrace.
\end{equation*}
\begin{prop}\label{noyau de neumann}
For any integer $k\geq1$, $\mathcal{N}_{p,k}^{\Delta}$ the subspace of all functions in $W_{k}^{1,p}(\Omega)$ of the form $w(p)-p$, where $p$ runs over all polynomials of $\mathcal{P}_{[1-3/p-k]}^{\Delta}$ and $w(p)$ is the unique solution in $W_{0}^{1,2}(\Omega)\cap W_{k}^{1,p}(\Omega)$ of the Neumann problem
\begin{equation}\label{problem de Neumann liee au noyeau}
\Delta\,w(p)=0\quad\mathrm{in}\,\, \Omega\quad\mathrm{and}\quad\dfrac{\partial w(p)}{\partial \textbf{\textit{n}}}=\dfrac{\partial p}{\partial \textbf{\textit{n}}}\quad\mathrm{on}\,\, \Gamma.
\end{equation}
Here also, we set $\mathcal{N}_{p,k}^{\Delta}=\left\lbrace 0\right\rbrace $ when $k\leq0$; $\mathcal{N}_{p,k}^{\Delta}$ is a finite-dimentional space of the same dimension as $\mathcal{P}_{[1-3/p-k]}^{\Delta}$ and in particular, $\mathcal{N}_{0}^{\Delta}=\R$.
\end{prop}
\noindent The next theorem states an existence, uniqueness and regularity result for problem~\eqref{Problem de Neumann Ax}, for $k=0$ and $1<p<\infty$.
\begin{theo}\label{theorem Neumann}
For any $\textbf{\textit{f}}$ in $L^{p}(\Omega)$ and $g$ in $W^{-1/p,p}(\Gamma)$. Then, the problem~\eqref{Problem de Neumann Ax} has a solution $\textbf{\textit{u}}\in W^{1,p}_{-1}(\Omega)$  unique up to element of $\mathcal{N}^{\Delta}_{-1,p}(\Omega)$ and we have the following estimate: 
\begin{eqnarray}\label{estimation faible neumann}
\Vert\textbf{\textit{u}}\Vert_{W^{1,p}_{-1}(\Omega)/ \mathcal{N}^{\Delta}_{-1,p}(\Omega)} \leqslant C\big(\Vert\textbf{\textit{f}}\Vert_{L^{p}(\Omega)}+\Vert g\Vert_{W^{-1/p,p}(\Omega)}\big).
\end{eqnarray}
If in addition, $g$ in $W^{1/p',p}(\Gamma)$, the solution $\textbf{\textit{u}}$ of problem~\eqref{Problem de Neumann Ax} belongs to $W^{2,p}_{0}(\Omega)$ and satisfies 
\begin{eqnarray}\label{estimation forte neumann}
\Vert\textbf{\textit{u}}\Vert_{W^{2,p}_{0}(\Omega)} \leqslant C\big(\Vert\textbf{\textit{f}}\Vert_{L^{p}(\Omega)}+\Vert g\Vert_{W^{1/p',p}(\Omega)}\big).
\end{eqnarray}
\end{theo}
\begin{proof}
Let us extend $\textbf{\textit{f}}$ by zero in $\Omega'$ and let $\widetilde{\textbf{\textit{f}}}$ denote the extended function. Then $\widetilde{\textbf{\textit{f}}}$ belongs to $L^{p}(\R^3)$. Applying \cite{Amrouche_1994}, there exists a unique function $\widetilde{\textbf{\textit{v}}} \in W^{2,p}_{0}(\R^3)/\mathcal{P}^{\Delta}_{[2-3/p]}$ such that
\begin{eqnarray}\label{laplace in R^3}
-\Delta\,\widetilde{\textbf{\textit{v}}}=\widetilde{\textbf{\textit{f}}}\quad\text{ in }\,\,\R^3.
\end{eqnarray}
Then $\nabla\,\widetilde{\textbf{\textit{v}}}\cdot\textbf{\textit{n}}$ belongs to $W^{1-1/p,p}(\Gamma)\hookrightarrow W^{-1/p,p}(\Gamma)$. It follows from \cite[Theorem 3.12]{Louati_Meslameni_Razafison}, that the following problem:
\begin{eqnarray}\label{Neumann Harmonic}
\Delta\,\textbf{\textit{w}}=0\quad\text{ in }\,\,\,\Omega\quad\text{ and } \quad\nabla\,\textbf{\textit{w}}\cdot\textbf{\textit{n}}=g-\nabla\,\widetilde{\textbf{\textit{v}}}\cdot\textbf{\textit{n}}\quad\text{ on }\,\,\,\Gamma,
\end{eqnarray}
has a solution $w\in W^{1,p}_{-1}(\Omega)$ unique up to element of $\mathcal{N}^{\Delta}_{-1,p}(\Omega)$. Thus $\textbf{\textit{u}}=\widetilde{\textbf{\textit{v}}}_{\mid_{\Omega}}
+\textbf{\textit{w}}\in W^{1,p}_{-1}(\Omega)$ is the required solution of ~\eqref{Problem de Neumann Ax}. The uniqueness follows immediately from Proposition~\ref{noyau de neumann} .\\

\noindent Now, suppose that $g$ belongs to $W^{1/p',p}(\Gamma).$ The aim is to  prove that $\textbf{\textit{u}}$ belongs to $W^{2,p}_{0}(\Omega)$. To that end, let us introduce the following partition of unity:
\begin{equation}
 \label{partition de l'unite}
 \begin{split}
&\varphi,\,\psi\in\mathcal{C}^\infty(\R^3),\quad 0\le\varphi,\,\psi\le1,\quad \varphi+\psi=1\quad\text{in}\quad\R^3,\\ 
&\varphi=1\quad\text{in}\quad B_R,\quad\text{supp }\varphi\subset B_{R+1}.
\end{split}
\end{equation}
Let $P$ be a continuous linear mapping from $W^{1,p}_{-1}(\Omega)$ to $W^{1,p}_{-1}(\R^3)$, such that $P\,\textbf{\textit{u}}=\widetilde{\textbf{\textit{u}}}$. Then $\widetilde{\textbf{\textit{u}}}$ belongs to $W^{1,p}_{-1}(\R^3)$ and can be written as:
\begin{eqnarray*}
\widetilde{\textbf{\textit{u}}}=\varphi\,{\widetilde{\textbf{\textit{u}}}}+\psi\,\widetilde{\textbf{\textit{u}}}.
\end{eqnarray*}
Next, one can easily observe that $\widetilde{\textbf{\textit{u}}}$ satisfies the following problem:
\begin{eqnarray}\label{laplace R^3}
-\Delta\,\psi\,\widetilde{\textbf{\textit{u}}}=\textbf{\textit{f}}_{1}\quad\text{ in }\,\,\,\R^3,
\end{eqnarray}
with
\begin{eqnarray*}
\textbf{\textit{f}}_{1}=\widetilde{\textbf{\textit{f}}}{\psi}
-(2\nabla{\widetilde{\textbf{\textit{u}}}}\nabla{\psi}+\widetilde{\textbf{\textit{u}}}\Delta\,{\psi}).
\end{eqnarray*}
Owing to the support of $\psi$, $\textbf{\textit{f}}_{1}$ has the same regularity as $\textbf{\textit{f}}$ and so belongs to $L^{p}(\R^3)$. It follows from \cite{Amrouche_1994}, that there exists $\z$ in $W^{2,p}_{0}(\R^3)$ such that $-\Delta\,\z=\textbf{\textit{f}}_{1}$ in $\R^3.$ This implies that $\psi\,\widetilde{\textbf{\textit{u}}}-\z$ is a harmonic tempered distribution and therefore a harmonic polynomial that belongs to $\mathcal{P}^{\Delta}_{[2-3/p]}$. The fact that $\mathcal{P}^{\Delta}_{[2-3/p]} \subset  W^{2,p}_{0}(\R^3)$ yields that ${\psi\,\widetilde{\textbf{\textit{u}}}}$ belongs to $W^{2,p}_{0}(\R^3)$. In particular, we have ${\psi\,\widetilde{\textbf{\textit{u}}}}=\textbf{\textit{u}}$ outside $B_{R+1}$, so the restriction of $\textbf{\textit{u}}$ to $\partial\, B_{R+1}$ belongs to $W^{2-1/p,p}(\partial\,B_{R+1})$. Therefore, $\varphi\,\textbf{\textit{u}}\in W^{1,p}(\Omega_{R+1})$ satisfies the following problem:
\begin{eqnarray}\label{mixte laplace}
\begin{cases}
-\Delta\,\varphi\,\textbf{\textit{u}}=\textbf{\textit{f}}_{2}\quad\text{ in }\,\,\, \Omega_{R+1},\\
\nabla\,\varphi\,\textbf{\textit{u}}\cdot\textbf{\textit{n}}=g\quad\text{ on }\,\,\,\Gamma,\\
\varphi\,\textbf{\textit{u}}={\psi\,\widetilde{\textbf{\textit{u}}}}\quad\text{ on }\,\,\,\partial\, B_{R+1},
\end{cases}
\end{eqnarray}
where $\textbf{\textit{f}}_{2}\in L^{p}(\Omega_{R+1})$ have similar expression as  $\textbf{\textit{f}}_{1}$ with $\psi$ remplaced by $\varphi$. According [Remark 3.2, \cite{Amrouche_JMPA_1997}], $\varphi\,\textbf{\textit{u}}\in W^{2,p}(\Omega_{R+1})$, which in turn shows that $\varphi\,\widetilde{\textbf{\textit{u}}}$ also belongs to $W^{2,p}(\Omega_{R+1})$. This implies that $\textbf{\textit{u}}\in W^{2,p}_{0}(\Omega).$\\
\end{proof}


\subsection{A mixed Stokes problem}\quad\\
Let $R>0$ be a real number large enough so that $\overline{\Omega'}\subset B_R$. We recall that $\Omega_R=\Omega \cap B_R$ and $\partial\,B_{R}=\Gamma_{R}$.  Now, we study the following mixed boundary value problem: Given $\textbf{\textit{f}}$, $\chi$, $g$, $\textbf{\textit{h}}$ and $\textbf{\textit{a}}$. We look for $(\textbf{\textit{u}},\pi)$ satisfying

\begin{eqnarray}\label{Problem oxiliere }
\begin{cases}
-\Delta\,\textbf{\textit{u}}+\nabla\,{\pi}=\textbf{\textit{f}}\quad and\quad
\mathrm{div}\,\textbf{\textit{u}}={\chi}\qquad\quad\text{ in }\,\Omega_{R},\\
\quad\textbf{\textit{u}}\cdot \textbf{\textit{n}}=g\,\,\,\,and\quad
2[\mathrm{\textbf{D}}(\textbf{\textit{u}})\textbf{\textit{n}}]_{\tau}+\alpha{\textbf{\textit{u}}}_{\tau}=\textbf{\textit{h}}\quad\text{ on }\,\Gamma,\\
\quad\textbf{\textit{u}}=\textbf{\textit{a}}\,\,\, \text{on}\,\,\Gamma_{R}.
\end{cases}
\end{eqnarray}
\begin{theo}\label{problem mixte} Assume that $1< p < \infty$. Let the pair $(\textbf{\textit{u}},\pi)\in H^{2}(\Omega_R)\times H^{1}(\Omega_R)$ be a solution to the problem~\eqref{Problem oxiliere } with data $\textbf{\textit{f}}\in L^p(\Omega_R)$, $\chi\in W^{1,p}(\Omega_R)$, $g\in W^{1+1/p',p}(\Gamma)$, $\textbf{\textit{h}}\in W^{1-1/p,p}(\Gamma)$, $\textbf{\textit{a}}\in W^{1+1/p',p}(\partial B_R)$ such that $\textbf{\textit{h}}\cdot\textbf{\textit{n}}=0$ on $\Gamma$. Then, we also have $\textbf{\textit{u}}\in W^{2,p}(\Omega_R)$ and $\pi\in W^{1,p}(\Omega_R)$.
\end{theo}
%
\begin{proof}
The proof of the theorem is made of tow steps.
\begin{itemize}
\item[Step 1.] The case $ 1< p\leqslant 2$.\\
Since $(\textbf{\textit{u}},\pi)\in H^{2}(\Omega_R)\times H^{1}(\Omega_R)$ then its is clear that $(\textbf{\textit{u}},\pi)\in W^{2,p}(\Omega_R)\times W^{1,p}(\Omega_R)$ for any $ 1< p\leqslant 2$.

\item[Step 2.] The case $2< p<\infty$.\\ 	Let us introduce the following partition of unity:
\begin{equation*}
 \begin{split}
&\theta_{1},\,\theta_{2}\in\mathcal{C}^\infty(\overline{\Omega}_R),\quad 0\le\theta_{1},\,\theta_{2}\le1,\quad \theta_{1}+\theta_{2}=1\quad\text{in}\quad\Omega_R,\\ 
&\theta_{1}=1\quad\text{in}\quad B_{R/3},\quad\text{supp }\theta_{1}\subset B_{2R/3}.
\end{split}
\end{equation*}
Then we can write $\textbf{\textit{u}}=\theta_{1}\textbf{\textit{u}}+\theta_{2}\textbf{\textit{u}}=\textbf{\textit{u}}_{1}+\textbf{\textit{u}}_{2}$
and $\pi=\theta_{1}\pi+\theta_{2}\pi=\pi_1+\pi_2$. Now it is clear that
\begin{eqnarray}\label{f_1 et x_1 dans omega_R}
-\Delta\,\textbf{\textit{u}}_{1}+\nabla\,\pi_{1}=\textbf{\textit{f}}_{1}\quad\text{ and }\quad \mathrm{div}\,\textbf{\textit{u}}_{1}=\chi_{1}\quad\text{ in }\,\,\Omega_R,
\end{eqnarray}
where $\textbf{\textit{f}}_{1}=\theta_{1}\,\textbf{\textit{f}}-\textbf{\textit{u}}\,\Delta\,\theta_{1}-2\nabla\,\textbf{\textit{u}}\cdot\nabla\,\theta_{1}+\pi\,\nabla\,\theta_{1}\in L^{p}(\Omega_R)$ and $\chi_{1}=\theta\,\chi+\textbf{\textit{u}}\,\nabla\,\theta_{1}\in W^{1,p}(\Omega_R)$ for $2\leqslant p\leqslant 6$. It is also clear that on the boundaries, we have
\begin{eqnarray}\label{h sur Gamma}
\begin{cases}
\textbf{\textit{u}}_{1}\cdot\textbf{\textit{n}}=g\quad,\quad 2[\mathrm{\textbf{D}}(\textbf{\textit{u}}_{1})\textbf{\textit{n}}]_{\tau}+\alpha{\textbf{\textit{u}}_{1}}_{\tau}=\textbf{\textit{h}}\quad\text{ on }\,\,\,\Gamma,\\
\textbf{\textit{u}}_{1}\cdot\textbf{\textit{n}}=0\quad,\quad 2[\mathrm{\textbf{D}}(\textbf{\textit{u}}_{1})\textbf{\textit{n}}]_{\tau}+\alpha{\textbf{\textit{u}}_{1}}_{\tau}=\textbf{\textit{0}}\quad\text{ on }\,\,\,\partial B_R,
\end{cases}
\end{eqnarray}
where $\textbf{\textit{h}}\in W^{1-1/p,p}(\Gamma)$ and $g\in W^{1+1/p',p}(\Gamma)$, we deduce that $(\textbf{\textit{u}}_{1},\pi_{1})$ that satisfies \eqref{f_1 et x_1 dans omega_R}-\eqref{h sur Gamma} belongs to $W^{2,p}(\Omega_R)\times W^{1,p}(\Omega_R)$ for $2< p\leqslant 6$ see \cite[Theorem 2.1]{Amrita_2018}.\\
Similar arguments show $(\textbf{\textit{u}}_{2},\pi_{2})$ satisfies the following Stokes problem with the Dirichlet boundary 
\begin{eqnarray}\label{maria}
\begin{cases}
-\Delta\,\textbf{\textit{u}}_{2}+\nabla\,\pi_{2}=\textbf{\textit{f}}_{2}\quad\text{ and }\quad \mathrm{div}\,\textbf{\textit{u}}_{2}=\chi_{2}\quad\text{ in }\,\,\Omega_R,\\
\textbf{\textit{u}}_{2}=0\quad\text{ on }\,\,\Gamma\quad\text{ and }\quad \textbf{\textit{u}}_{2}=\textbf{\textit{a}}\quad\text{ on }\,\,\partial B_R. 
\end{cases}
\end{eqnarray}
Where $\textbf{\textit{f}}_{2}\in L^{p}(\Omega_R)$ and $\chi_{2}\in W^{1,p}(\Omega_R)$ have similar expression as $\textbf{\textit{f}}_{1}$ and $\chi_{1}$ with $\theta_1$ replaced by $\theta_2$. Since $\textbf{\textit{a}}\in W^{1+1/p',p}(\partial B_R)$, the problem~\eqref{maria} has a solution  $(\textbf{\textit{u}}_{2},\pi_2)$ belongs to $W^{2,p}(\Omega_R)\times W^{1,p}(\Omega_R)$ (see for instance~\cite{Amrouche_CMJ_1994} or~\cite{Galdi_book_1}).\\
Now, suppose that $p>6$. The above argument shows that $(\textbf{\textit{u}},\pi)$ belongs to $W^{2,6}(\Omega_R)\times W^{1,6}(\Omega_R)$ and we can repeat the same argument with $p=6$ instead $p=2$. We have $W^{1,6}(\mathcal{O})\hookrightarrow L^{q}(\mathcal{O})$ for any real number $q>6$. We know that the embedding
\begin{center}
$W^{2,p}(\Omega')\hookrightarrow W^{1,q}(\Omega')$,
\end{center}
for any $q\in [1,\infty]$ if $p>3$. Then we have $\textbf{\textit{u}}_{\mid_{\Gamma}}\in W^{2-1/6,6}(\Gamma)\hookrightarrow W^{1-1/p}(\Gamma)$
 for all $p>6$. Consequently, by the some reasoning we deduce that the solution  $(\textbf{\textit{u}},\pi)$ belongs to $W^{2,p}(\Omega_R)\times W^{1,p}(\Omega_R)$.
 \end{itemize}
\end{proof}


\section{Strong solutions for the exterior Stokes problem}\label{sect.string.solution}
In this section, we are interested in the following problem:
\begin{eqnarray*}
(\mathcal{S}_{T})
\begin{cases}
-\Delta\,\textbf{\textit{u}}+\nabla\,{\pi}=\textbf{\textit{f}}\quad and\quad
\mathrm{div}\,\textbf{\textit{u}}={\chi}\qquad\quad\text{ in }\,\Omega,\\
\quad\textbf{\textit{u}}\cdot \textbf{\textit{n}}=g\,\,\,\,and\quad
2[\mathrm{\textbf{D}}(\textbf{\textit{u}})\textbf{\textit{n}}]_{\tau}+\alpha\textbf{\textit{u}}_{\tau}=\textbf{\textit{h}}\quad\text{ on }\,\Gamma.

\end{cases}
\end{eqnarray*}

In this part, we investigate the well-posedness of strong solutions in $W_{k+1}^{2,p}(\Omega)\times W_{k+1}^{1,p}(\Omega)$ with $k\in \Z$. In order to deal with the uniqueness issues, we first need to characterize the kernel of problem $(\mathcal{S}_{T})$. We define the kernel of problem $(\mathcal{S}_T)$.
\noindent For $k\in\Z$ and $p\geqslant 2$, we introduce:
\begin{eqnarray*}
&&\mathcal{N}_{k+1}^{p}(\Omega)=\Big\lbrace  (\textbf{\textit{u}},{\pi})\in W^{2,p}_{k+1}(\Omega)\times{{W}}^{1,p}_{k+1}(\Omega);\,\\
&& -\Delta\,\textbf{\textit{u}}+\nabla\,{\pi}=0,\, \mathrm{div}\,\textbf{\textit{u}}=0\text{ in }\,\Omega\text{ and }\textbf{\textit{u}}\cdot \textbf{\textit{n}}=0,\,\,2[\mathrm{\textbf{D}}(\textbf{\textit{u}})\textbf{\textit{n}}]_{\tau}+\alpha\textbf{\textit{u}}_{\tau}=0\,\text{ on }\,\Gamma\Big\rbrace.
\end{eqnarray*}

\noindent The characterization of the kernel $\mathcal{N}_{k+1}^{p}(\Omega)$ is given by the following proposition:
\begin{prop}\label{caracterisation du noyau}
Consider an exterior domain $\Omega$ with boundary $\Gamma$ of class $\mathcal{C}^{2.1}$. We have\\
 $$\mathcal{N}_{k+1}^{p}(\Omega)=\left\lbrace  (\textbf{\textit{v}}-\boldsymbol{\lambda},{\theta}-\mu); \quad (\boldsymbol{\lambda},\mu)\in N_{[1-k-3/p]}\right\rbrace, $$ 
where $(\textbf{\textit{v}},{\theta})\in W^{2,p}_{k+1}(\Omega)\cap W^{2,2}_{1}(\Omega)\times{{W}}^{1,p}_{k+1}(\Omega)\cap W^{1,2}_{1}(\Omega)$ is the unique solution of the following problem:
\begin{equation}
\label{problem stokes du noyau}
\begin{cases}
-\Delta\,\textbf{\textit{v}}+\nabla\,{\theta}=\boldsymbol{0}\quad\text{and}\quad
\mathrm{div}\,\textbf{\textit{v}}=0\quad\text{in}\quad\Omega,\\[4pt]
\textbf{\textit{v}}\cdot \textbf{\textit{n}}=\boldsymbol{\lambda}\cdot\textbf{\textit{n}}\quad\text{and}\quad
2[\mathrm{\textbf{D}}(\textbf{\textit{v}})\textbf{\textit{n}}]_{\tau}+\alpha\textbf{\textit{v}}_{\tau}=
2[\mathrm{\textbf{D}}(\boldsymbol{\lambda})\textbf{\textit{n}}]_{\tau}+\alpha\boldsymbol{\lambda}_{\tau}\quad\text{on}\quad\Gamma.
\end{cases}
\end{equation}
In particular, $\mathcal{N}_{k+1}^{p}(\Omega)=\lbrace(\textbf{\textit{0}},0)\rbrace$ if $k>1-3/p$.\\ 
\end{prop}
\begin{proof}
Let us assume that $(\textbf{\textit{u}},\pi)$ belongs to $\mathcal{N}_{k+1}^{p}(\Omega)$. 
The pair $(\textbf{\textit{u}},\pi)$ has an extension $(\widetilde{\textbf{\textit{u}}},\widetilde{\pi})$ that belongs to $W^{2,p}_{k+1}(\R^3)\times{{W}}^{1,p}_{k+1}(\R^3)$. 
Set now
$$
\textbf{\textit{F}}=-\Delta\,\widetilde{\textbf{\textit{u}}}+\nabla\,\widetilde{{\pi}}\quad\text{and}\quad e=\mathrm{div}\,\widetilde{\textbf{\textit{u}}}.
$$
Then the pair $(\textbf{\textit{F}},e)$ belongs to $W^{0,p}_{k+1}(\R^3)\times W^{1,p}_{k+1}(\R^3)$ and has a compact support. Therefore $(\textbf{\textit{F}},e)$ also belongs to $W^{0,2}_{1}
(\R^3)\times W^{1,2}_{1}(\R^3)$. It follows from [Theorem 3.9,\cite{Alliot_M3AS_1999}], that there exists a unique solution $(\textbf{\textit{v}},{\theta})\in W^{2,2}_{1}(\R^3)\times W^{1,2}_{1}(\R^3)$
 such that
$$-\Delta\,\textbf{\textit{v}}+\nabla\,{\theta}=-\Delta\,\widetilde{\textbf{\textit{u}}}+\nabla\,\widetilde{{\pi}}\quad\text{and}\quad
\mathrm{div}\,\textbf{\textit{v}}=\mathrm{div}\,\widetilde{\textbf{\textit{u}}}\quad\text{in}\quad\R^3 .$$

\noindent It follows that $({\textbf{\textit{v}}}-\widetilde{\textbf{\textit{u}}},{\theta}-\widetilde{\pi})$ belongs to 
$\left(W^{2,2}_{1}(\R^3)+{{W}}^{2,p}_{k+1}(\R^3) \right) \times 
\left(W^{1,2}_{1}(\R^3)+ W^{1,p}_{k+1}(\R^3) \right)$. Hence, there exits $(\boldsymbol{\lambda},\mu)\in N_{[1-k-3/p]}$ such that $({\textbf{\textit{v}}}-\widetilde{\textbf{\textit{u}}},{\theta}-\widetilde{\pi})=(\boldsymbol{\lambda},\mu)$,
Thus, $(\textbf{\textit{v}},\theta)$ belongs to $\left(W^{2,2}_{1}(\R^3)\cap{{W}}^{2,p}_{k+1}(\R^3) \right) \times 
\left(W^{1,2}_{1}(\R^3)\cap W^{1,p}_{k+1}(\R^3)\right)$,  its restriction to $\Omega$ belongs to $\left(W^{2,2}_{1}(\Omega)\cap{{W}}^{2,p}_{k+1}(\Omega) \right) \times 
\left(W^{1,2}_{1}(\Omega)\cap W^{1,p}_{k+1}(\Omega)\right)$ and satisfies \eqref{problem stokes du noyau}.\\

\end{proof}
 

\noindent The next Theorem solves the problem $(\mathcal{S}_{T})$ when $p\geqslant 2$, our study is based on strong solutions in a Hilbertian framework (see~\cite{DMR-2019}), that's why we will take the data $\textbf{\textit{f}}$ and $\chi$ have a compact support.

\begin{theo}\label{solution forte f et xi a support compact}
Assume that $p\geqslant 2$. Let $\textbf{\textit{f}}\in W^{0,p}_{k+1}(\Omega)$, $\chi\in W^{1,p}_{k+1}(\Omega)$, $g\in W^{1+1/p',p}(\Gamma)$ and $\textbf{\textit{h}}\in W^{1-1/p,p}(\Gamma)$ such that $\textbf{\textit{f}}$ and $\chi$ have a compact support, $\textbf{\textit{h}}\cdot\textbf{\textit{n}}=0$ on $\Gamma$ and the following compatibility condition is satisfied 
\begin{equation}\label{CC TH-strong}
 \forall (\boldsymbol{\xi},\eta)\in \mathcal{N}_{-k+1}^{p'}(\Omega),\quad\quad
\displaystyle\int_{\Omega} \textbf{\textit{f }}\cdot\boldsymbol{\xi}d\textbf{\textit{x}}-\displaystyle\int_{\Omega}
\chi\,{\eta}d\textbf{\textit{x}}=\int_{\Gamma} g\Big( 2[\textrm{\textbf{D}}(\boldsymbol{\xi})\textbf{\textit{n}}]\cdot\textbf{\textit{n}}-\eta \Big)d\boldsymbol{\sigma}-\int_{\Gamma}\textbf{\textit{h}}\cdot
\boldsymbol{\xi}d\boldsymbol{\sigma}. 
\end{equation}

\noindent Then, the Stokes problem $(\mathcal{S}_{T})$ has a solution $(\textbf{\textit{u}},{\pi})\in W^{2,p}_{k+1}(\Omega)\times {{W}}^{1,p}_{k+1}(\Omega)$ unique up to an element of $\mathcal{N}_{k+1}^{p}(\Omega)$. In addition, we have the following estimate:
 
\begin{eqnarray*}\label{esti}
&&\inf _{(\boldsymbol{\lambda},\mu)\in\mathcal{N}_{k+1}^{p}(\Omega)} \left( \Vert \textbf{\textit{u}}+\boldsymbol{\lambda}\Vert_{W^{2,p}_{k+1}(\Omega)}
+\Vert {\pi}+{\mu}\Vert_{{W}^{1,p}_{k+1}(\Omega)}\right) \\ 
&&\leqslant C\Big(\Vert\textbf{\textit{f}}\Vert_{W^{0,p}_{k+1}(\Omega)}+\Vert\chi\Vert_{W^{1,p}_{k+1}(\Omega)}+\Vert \textbf{\textit{h}}\Vert_{W^{1-1/p,p}(\Gamma)}+\Vert g\Vert_{W^{1+1/p',p}(\Gamma)}\Big).
\end{eqnarray*}
\end{theo}
\begin{proof}
Observe first that the uniqueness is a straightforward consequence of Proposition~\ref{caracterisation du noyau}. We now divide the proof of the theorem into 
several parts.\\

$\bullet$ \textsc{Compatibility condition}. In this part, we prove that~\eqref{CC TH-strong} is a necessary condition. 

Let $(\boldsymbol{\xi},\eta)$ be in $\mathcal{N}^{p'}_{-k+1}(\Omega)$. For any ($\boldsymbol{\varphi},\psi)\in\mathcal{D}(\overline{\Omega})\times\mathcal{D}(\overline{\Omega}).$ Using the same calculation as in the proof of \cite[Theorem 3.7]{DMR-2019}, we have 

\begin{eqnarray}\label{CC Green's formular p}
\nonumber &&\displaystyle\int_{\Omega}\big[\big(-\Delta\,\boldsymbol{\varphi}+\nabla\,\psi\big)\cdot\boldsymbol{\xi}-\eta\,\mathrm{div}\,\boldsymbol{\varphi}\big] d\textbf{\textit{x}}\\
&&=
\int_{\Gamma}({\boldsymbol{\varphi}\cdot\textbf{\textit{n}}})\Big(2[\mathrm{\textbf{D}}(\boldsymbol{\xi})\textbf{\textit{n}}]\cdot{\textbf{\textit{n}}}-\eta \Big)d\boldsymbol{\sigma}
-\int_{\Gamma} \boldsymbol{\xi}_{\tau}\cdot\Big(2[\mathrm{\textbf{D}}(\boldsymbol{\varphi})\textbf{\textit{n}}]_{\tau}+\alpha\,\boldsymbol{\varphi}_{\tau}\Big)d\boldsymbol{\sigma}.
\end{eqnarray}

Then, the last Green's formula 
holds for any pair $(\boldsymbol{\varphi},\psi)\in W^{2,p}_{k+1}(\Omega)\times W^{1,p}_{k+1}(\Omega)$ by density.
In particular, if $(\textbf{\textit{u}},\pi)\in W^{2,p}_{k+1}(\Omega)\times W^{1,p}_{k+1}(\Omega)$ is a solution of $(\mathcal{S}_{T})$, then~\eqref{CC TH-strong} holds.\\
\item[$\bullet$] \textsc{Existence}. Here we prove that problem ($\mathcal{S}_T$) has a solution $(\textit{\textbf{u}},\pi)$ that belongs to $ W_k^{1,p}(\Omega)\times W_k^{0,p}(\Omega)$. We start with the case $2\leqslant p\leqslant 6$, the proof is made of two steps.
\vspace*{-1cm}

\vspace*{0.95cm}
\item[Step 1.] The case $g=0$.\\
\noindent Since  $p\geqslant 2$ and $(\textbf{\textit{f}},\chi)$ have a support compact, then we have $(\textbf{\textit{f}},\chi)$ belongs to $W^{0,2}_{1}(\Omega)\times W^{1,2}_{1}(\Omega)$. In addition, its clear that $\textbf{\textit{h}}$ belongs to $H^{1/2}_{}(\Gamma)$.
Thanks to \cite[Theorem 3.6]{DMR-2019}, problem $(\mathcal{S}_{T})$ has a solution $(\textbf{\textit{u}},{\pi})\in W^{2,2}_{1}(\Omega)\times W^{1,2}_{1}(\Omega)$. 
It remains now to prove that $(\textbf{\textit{u}},{\pi})$ belongs to $W^{1,p}_{k}(\Omega)\times {W}^{0,p}_{k}(\Omega)$.
To that end, we shall use again properties of the Stokes problem in the whole space $\R^3$. Now, we first need appropriate extensions of $\textbf{\textit{u}}$ and $\pi$
defined in $\R^3$. So let us consider the following Stokes problem in the bounded domain $\Omega'$
\begin{equation}
\label{probleme borne omega'}
\begin{cases}
-\Delta\textbf{\textit{u'}}+\nabla\pi'=\boldsymbol{0}\quad\text{and}\quad\mathrm{div}\,\textbf{\textit{u'}}=0\quad\text{in}\quad\Omega',\\
\textbf{\textit{u'}}=\textbf{\textit{u}}\quad\text{on}\quad\Gamma.
\end{cases}
\end{equation}
Since $\textit{\textbf{u}}\cdot\textit{\textbf{n}}=0$, problem~\eqref{probleme borne omega'} has a solution $(\textit{\textbf{u'}},\pi')\in H^2(\Omega')\times H^1(\Omega')$
(see for instance~\cite{Amrouche_CMJ_1994} or~\cite{Galdi_book_1}).
\noindent Setting
\begin{eqnarray*}\widetilde{\textbf{\textit{u}}}=
\begin{cases}
{\textbf{\textit{u}}}\quad\text{in}\quad\Omega,\\
{\textbf{\textit{u'}}}\quad\text{in}\quad\Omega',\\
\end{cases}
\quad\text{and}\quad\quad
\widetilde{\pi}=
\begin{cases}
\pi\quad\text{in}\quad\Omega,\\
\pi'\quad\text{in}\quad\Omega'.\\
\end{cases}
\end{eqnarray*}
Then clearly, the pair $(\widetilde{\textbf{\textit{u}}},\widetilde{\pi})$ belongs to $ W^{1,2}_1(\R^3)\times {W}^{0,2}_{1}(\R^3)$. 
The goal is now to identify $(\widetilde{\textbf{\textit{u}}},\widetilde{\pi})$ with a solution of the Stokes problem in $\R^3$
that belongs to $W^{1,p}_{k}(\Omega)\times {W}^{0,p}_{k}(\Omega)$.
Let us set 
\begin{equation}
\label{def.F.e}
\textbf{\textit{F}}=-\Delta\,\widetilde{\textit{\textbf{u}}}+\nabla\widetilde{\pi}\quad\text{and}\quad e=\mathrm{div}\,\widetilde{\textit{\textbf{u}}}.
\end{equation}
 In order to apply Theorem~\ref{theorem de Stokes dans R^3}, we need to show that $(\textbf{\textit{F}},e)$ belongs to $W^{-1,p}_{k}(\R^3)\times W^{0,p}_{k}(\R^3)$ and satisfies~\eqref{CCR^3}. Therefore, denoting by $\widetilde{\textit{\textbf{f }}}\in W_{k+1}^{0,p}(\R^3)$ the extension of $\textit{\textbf{f }}$ by zero in $\Omega'$ , we deduce that for any 
$\boldsymbol{\varphi}\in \mathcal{D}(\R^3)$
\begin{eqnarray}
\label{formule de Green pour le prolongement F}
&&\langle\textbf{\textit{F}},\boldsymbol{\varphi}\rangle_{\boldsymbol{\mathcal{D'}}(\R^3)\times\mathcal{D}(\R^3)} 
= \displaystyle\int_{\R^3} \widetilde{\textbf{\textit{f}}}\cdot\boldsymbol{\varphi}d\textbf{\textit{x}}+2\int_{\Gamma}\boldsymbol{\varphi}\cdot\Big(\mathrm{\textbf{D}}({\textbf{\textit{u}}})\textbf{\textit{n}}
-\mathrm{\textbf{D}}({\textbf{\textit{u'}}})\textbf{\textit{n}}\Big)d\boldsymbol{\sigma}\\
\nonumber &&+\int_{\Gamma}\big(\boldsymbol{\varphi}\cdot\textbf{\textit{n}}\big)\big(\pi'-\pi\big)d\boldsymbol{\sigma}
- \int_{\Gamma}(\boldsymbol{\varphi}\cdot\textbf{\textit{n}})\chi d\boldsymbol{\sigma}.
\end{eqnarray}
This calculations are the same as in the previous proof of \cite[Theorem 3.7]{DMR-2019}.\\
\noindent Since  $\big(2\mathrm{\textbf{D}}({\textbf{\textit{u}}})\textbf{\textit{n}}
-2\mathrm{\textbf{D}}({\textbf{\textit{u'}}})\textbf{\textit{n}}\big)\mid_{\Gamma}$ and $\big(\pi'-\pi\big)\mid_{\Gamma}$ belongs to $H^{1/2}(\Gamma)$.
\begin{eqnarray*}
&&\Big\vert\langle\textbf{\textit{F}},\boldsymbol{\varphi}\rangle_{\boldsymbol{\mathcal{D'}}(\R^3)\times\mathcal{D}(\R^3)} 
\Big\vert \leqslant  \Big\vert\displaystyle\int_{\R^3} \widetilde{\textbf{\textit{f}}}\cdot\boldsymbol{\varphi}d\textbf{\textit{x}}\Big\vert +\Big\vert \langle\boldsymbol{\varphi}\cdot\textbf{\textit{n}},\chi\rangle_{H^{-1/2}(\Gamma)\times H^{1/2}(\Gamma)}\Big\vert\\
&&+\Big\vert 2\langle\mathrm{\textbf{D}}({\textbf{\textit{u}}})\textbf{\textit{n}}
-\mathrm{\textbf{D}}({\textbf{\textit{u'}}})\textbf{\textit{n}},\boldsymbol{\varphi}\rangle_{H^{1/2}(\Gamma)\times H^{-1/2}(\Gamma)}\Big\vert
+\Big\vert\langle\boldsymbol{\varphi}\cdot\textbf{\textit{n}},\pi'-\pi\rangle_{H^{-1/2}(\Gamma)\times H^{1/2}(\Gamma)}\Big\vert\\
&\leqslant & \Vert\widetilde{\textbf{\textit{f}}}\Vert_{W^{0,p}_{k+1}(\R^3)}\Vert\boldsymbol{\varphi}\Vert_{W^{0,p'}_{-2}(\R^3)}+\Vert\chi\Vert_{W^{1,2}_{k+1}(\Omega)}\Vert\boldsymbol{\varphi}\Vert_{H^{-1/2}(\Gamma)}\\
&&+\Vert 2\mathrm{\textbf{D}}({\textbf{\textit{u}}})\textbf{\textit{n}}
-2\mathrm{\textbf{D}}({\textbf{\textit{u'}}})\textbf{\textit{n}}\Vert_{H^{1/2}(\Gamma)}\Vert\boldsymbol{\varphi}\Vert_{H^{-1/2}(\Gamma)}+\Vert\boldsymbol{\varphi}\cdot\textbf{\textit{n}}\Vert_{H^{-1/2}(\Gamma)}\Vert\pi'-\pi\Vert_{H^{1/2}(\Gamma)}\\
&\leqslant & \Vert{\textbf{\textit{f}}}\Vert_{W^{0,p}_{k+1}(\Omega)}\Vert\boldsymbol{\varphi}\Vert_{W^{1,p'}_{-k}(\R^3)}+ C\Vert\boldsymbol{\varphi}\Vert_{H^{-1/2}(\Gamma)}
\end{eqnarray*}
As we have 
\begin{eqnarray*}
H^{1/2}(\Gamma)\hookrightarrow W^{-1/6,6}(\Gamma)\hookrightarrow W^{-1/p,p}(\Gamma), \text{for any } 2\leqslant p\leqslant 6
\end{eqnarray*}
 We obtain 
\begin{eqnarray*}
\Vert\boldsymbol{\varphi}\Vert_{H^{-1/2}(\Gamma)} &\leqslant & C\Vert\boldsymbol{\varphi}\Vert_{W^{1/p,p'}(\Gamma)}\\
&\leqslant & C\Vert\boldsymbol{\varphi}\Vert_{W^{1,p'}_{-k}(\Omega)}\\
&\leqslant & C\Vert\boldsymbol{\varphi}\Vert_{W^{1,p'}_{-k}(\R^3)}
\end{eqnarray*}
Then, we have
\begin{eqnarray*}
\Big\vert\langle\textbf{\textit{F}},\boldsymbol{\varphi}\rangle_{\boldsymbol{\mathcal{D'}}(\R^3)\times\mathcal{D}(\R^3)} 
\Big\vert \leqslant C\Vert\boldsymbol{\varphi}\Vert_{W^{1,p'}_{-k}(\R^3)}.
\end{eqnarray*}
Since $\mathcal{D}(\R^3)$ is dense in $W^{1,p'}_{-k}(\R^3)$ , 
then \eqref{formule de Green pour le prolongement F} is still valid for any
$\boldsymbol{\varphi} \in W^{1,p'}_{-k}(\R^3)$ which implies that $\textit{\textbf{F }}$ belongs to $W_{k}^{-1,p}(\R^3)$.\\
\noindent Now for any $\phi \in \mathcal{D}(\R^3)$, we have
\begin{align}\label{formule de Green pour le prolongement2}
\nonumber &\langle e,{\phi}\rangle_{\mathcal{D'}(\R^3)\times\mathcal{D}(\R^3)} 
=\langle\mathrm{div}\,\widetilde{\textbf{\textit{u}}},\phi\rangle_{\mathcal{D'}(\R^3)\times\mathcal{D}(\R^3)} 
=-\displaystyle\int_{\R^3}\widetilde{\textbf{\textit{u}}}\cdot\nabla\,\phi d\textbf{\textit{x}}\\
\nonumber &=\displaystyle\int_{\Omega}\mathrm{div}\,\textbf{\textit{u }}\,\phi d\textbf{\textit{x}}
+\displaystyle\int_{\Omega'}\mathrm{div}\,\textbf{\textit{u' }}\,\phi d\textbf{\textit{x}}-\int_{\Gamma} (\textbf{\textit{u}}
\cdot\textbf{\textit{n}})\phi d\boldsymbol{\sigma}+\int_{\Gamma} (\textbf{\textit{u'}}\cdot\textbf{\textit{n}})\phi d\boldsymbol{\sigma}\\
&=\displaystyle\int_{\Omega}\chi\,\phi d\textbf{\textit{x}}.
\end{align}
Since $\chi\in W^{1,p}_{k+1}(\Omega)$. Then, we have
\begin{eqnarray*}\label{estimation e}
\Big\vert \langle e,{\phi}\rangle_{\mathcal{D'}(\R^3)\times\mathcal{D}(\R^3)}\Big\vert\leqslant \Vert\chi\Vert_{W^{0,p}_{k}(\Omega)}\Vert\phi\Vert_{W^{0,p'}_{-k}(\Omega)}\leqslant C\Vert\phi\Vert_{W^{0,p'}_{-k}(\Omega)}.
\end{eqnarray*}
Due to the density of $\mathcal{D}(\R^3)$ in $W^{0,p'}_{-k}(\R^3)$, then~\eqref{formule de Green pour le prolongement2} is still valid for any $\phi \in W^{0,p'}_{-k}(\R^3)$, which implies that  $e$ belongs to $W^{0,p}_{k}(\R^3)$. As a result, we have proved that $(\textit{\textbf{F}},e)$ belongs to $W_k^{-1,p}(\R^3)\times W_k^{0,p}(\R^3)$.\\
\noindent Let us now prove that $\textit{\textbf{F}}$ and $e$ satisfy~\eqref{CCR^3}, which in view of~\eqref{formule de Green pour le prolongement F} and~\eqref{formule de Green pour le prolongement2}, amounts to prove that for any $(\boldsymbol{\lambda},\mu)$ in $N_{[1-3/p'+k]}$
\begin{align}
\label{condition de compatibilite sur F et e}
&\displaystyle\int_{\R^3} \widetilde{\textbf{\textit{f}}}\cdot\boldsymbol{\lambda}d\textbf{\textit{x}}-\displaystyle\int_{\Omega}\chi\,\mu d\textbf{\textit{x}}+
2\int_{\Gamma}\boldsymbol{\lambda}\cdot\Big(\textbf{D}(\textbf{\textit{u}})\textbf{\textit{n}}-\textbf{D}(\textbf{\textit{u'}})\textbf{\textit{n}}\Big)d\boldsymbol{\sigma}
\\
\nonumber &+\int_{\Gamma} \big(\boldsymbol{\lambda}\cdot\textbf{\textit{n}}\big)\big(\pi'-\pi\big) d\boldsymbol{\sigma}-\displaystyle\int_{\Gamma}(\boldsymbol{\lambda}\cdot\textbf{\textit{n}})\chi d\boldsymbol{\sigma}=0.
\end{align}
\noindent So let $(\boldsymbol{\lambda},\mu)\in N_{[1-3/p'+k]}$ and let $(\textit{\textbf{v}}(\boldsymbol{\lambda}),\theta(\boldsymbol{\lambda}))$ be in $W^{2,p'}_{-k+1}(\Omega)\cap W^{2,2}_{1}(\Omega)\times W^{1,p'}_{-k+1}(\Omega)\cap W^{1,2}_{1}(\Omega)$
such that the pair $(\textit{\textbf{v}}(\boldsymbol{\lambda})-\boldsymbol{\lambda},\theta(\boldsymbol{\lambda})-\mu)$ belongs to $\mathcal{N}^{p'}_{-k+1}(\Omega)$.\\
\noindent Now, for any  $(\textit{\textbf{v}},\theta)\in  W^{2,p'}_{-k+1}(\Omega)\cap W^{2,2}_{1}(\Omega)\times W^{1,p'}_{-k+1}(\Omega)\cap W^{1,2}_{1}(\Omega)$ such that $\mathrm{div}\,\textit{\textbf{v}}=0$, computations in $\Omega$ yields

\begin{eqnarray}\label{calcul.Omega1}
&&\int_{\Omega}\textit{\textbf{f}}\cdot\textit{\textbf{v }}(\boldsymbol{\lambda})\,d\textit{\textbf{x }}
-\int_\Omega\chi\theta(\boldsymbol{\lambda}) d\textit{\textbf{x }}\\
\nonumber &&-\int_{\Gamma}(\textit{\textbf{v }}(\boldsymbol{\lambda})\cdot\textit{\textbf{n }})\chi d\boldsymbol{\sigma}-2\int_{\Gamma}\textit{\textbf{u }}\cdot\textbf{D}(\textit{\textbf{v }}(\boldsymbol{\lambda}))\textit{\textbf{n }}d\boldsymbol{\sigma}\\
\nonumber &&+2\int_{\Gamma}\textit{\textbf{v }}(\boldsymbol{\lambda})\cdot\textbf{D}(\textit{\textbf{u }})\textit{\textbf{n }}d\boldsymbol{\sigma}-\int_{\Gamma}(\textit{\textbf{v }}(\boldsymbol{\lambda})\cdot\textit{\textbf{n }})\pi d\boldsymbol{\sigma}=0.
\end{eqnarray}

\noindent Now making the difference between~\eqref{calcul.Omega1} and~\eqref{CC TH-strong} with $\boldsymbol\xi=\textit{\textbf{v}}(\boldsymbol{\lambda})-\boldsymbol{\lambda}$
and $\eta=\theta(\boldsymbol{\lambda})-\mu$ and recalling that $\textit{\textbf{v }}(\boldsymbol{\lambda})\cdot\textit{\textbf{n }}=\boldsymbol{\lambda}\cdot\textit{\textbf{n }}$
on $\Gamma$, yields
 
\begin{eqnarray}
\label{calcul.Omega.CC}
  &&\int_{\Omega}\textit{\textbf{f}}\cdot\boldsymbol{\lambda}\,d\textit{\textbf{x }}
-\int_\Omega\chi\mu d\textit{\textbf{x }}-\int_{\Gamma}(\textit{\textbf{v }}(\boldsymbol{\lambda})-\boldsymbol{\lambda})\cdot\textit{\textbf{h }}d\boldsymbol{\sigma}\\
\nonumber &&-\int_{\Gamma}(\boldsymbol{\lambda}\cdot\textit{\textbf{n }})\chi d\boldsymbol{\sigma}-2\int_{\Gamma}\textit{\textbf{u }}\cdot\textbf{D}(\textit{\textbf{v }}(\boldsymbol{\lambda}))\textit{\textbf{n }}d\boldsymbol{\sigma}\\
\nonumber &&+2\int_{\Gamma}\textit{\textbf{v }}(\boldsymbol{\lambda})\cdot\textbf{D}(\textit{\textbf{u }})\textit{\textbf{n }}d\boldsymbol{\sigma}-\int_{\Gamma}(\textit{\textbf{v }}(\boldsymbol{\lambda})\cdot\textit{\textbf{n }})\pi d\boldsymbol{\sigma}=0.
\end{eqnarray}
\noindent Computations on $\Omega'$ yield 
 \begin{align*}
&\int_{\Omega'}\big(-\Delta\,{\textbf{\textit{u'}}}+\nabla\,{\pi'}\big)\cdot\boldsymbol{\lambda}d\textbf{\textit{x}}=0\\
&=-\int_{\Omega'}{\textbf{\textit{u'}}}\cdot\Delta\,\boldsymbol{\lambda} d\textbf{\textit{x}}+2\int_{\Gamma}\boldsymbol{\lambda}\cdot \mathrm{\textbf{D}}({\textbf{\textit{u'}}})\textbf{\textit{n}}d\boldsymbol{\sigma}
-2\int_{\Gamma}\textbf{\textit{u'}}\cdot\mathrm{\textbf{D}}(\boldsymbol{\lambda})\textbf{\textit{n}}d\boldsymbol{\sigma}-\int_{\Gamma}(\boldsymbol{\lambda}\cdot\textbf{\textit{n}})\pi'd\boldsymbol{\sigma}.
\end{align*}
The fact that $\textit{\textbf{u }}=\textit{\textbf{u' }}$ on $\Gamma$, implies
$$\int_{\Omega'}\textit{\textbf{u'}}\cdot\Delta\boldsymbol{\lambda}\,d\textit{\textbf{x }}
=\int_{\Omega'}\textit{\textbf{u'}}\cdot\nabla\mu\,d\textit{\textbf{x }}=-\int_{\Omega'}\mu\,\mathrm{div}\,\textit{\textbf{u'}}\,d\textit{\textbf{x }}
+\int_{\sigma}\mu(\textit{\textbf{u'}}\cdot\textit{\textbf{n }})d\boldsymbol{\sigma}=0$$
and we deduce that
\begin{equation}
\label{integration sur omega'}
2\int_{\Gamma}\boldsymbol{\lambda}\cdot \mathrm{\textbf{D}}({\textbf{\textit{u'}}})\textbf{\textit{n}}d\boldsymbol{\sigma}
-2\int_{\Gamma}\textbf{\textit{u'}}\cdot\mathrm{\textbf{D}}(\boldsymbol{\lambda})\textbf{\textit{n}}d\boldsymbol{\sigma}-\int_{\Gamma}(\boldsymbol{\lambda}\cdot\textbf{\textit{n}})\pi'd\boldsymbol{\sigma}=0.
\end{equation}

\noindent Combining~\eqref{calcul.Omega.CC} and~\eqref{integration sur omega'} yields
\begin{eqnarray}
\begin{split}
\label{calcul.omega.omega'}
&&\int_{\R^3} \widetilde{\textbf{\textit{f}}}\cdot\boldsymbol{\lambda}d\textbf{\textit{x}}
-\displaystyle\int_{\Omega}\chi\,\mu d\textbf{\textit{x}}-\int_{\Gamma}\chi(\boldsymbol{\lambda}\cdot\textbf{\textit{n}})d\boldsymbol{\sigma}+
\int_{\Gamma}(\boldsymbol{\lambda}\cdot\textbf{\textit{n}})(\pi'-\pi) d\boldsymbol{\sigma}\\
&&-2\int_{\Gamma}\boldsymbol{\lambda}\cdot\mathrm{\textbf{D}}({\textbf{\textit{u'}}})\textbf{\textit{n}}d\boldsymbol{\sigma}+2\int_{\Gamma}\textbf{\textit{u}}\cdot\big(\mathrm{\textbf{D}}(\boldsymbol{\lambda})\textbf{\textit{n}}-\mathrm{\textbf{D}}(\textbf{\textit{v}}(\boldsymbol{\lambda}))\textbf{\textit{n}}\big)d\boldsymbol{\sigma}
\\
&&+2\int_{\Gamma}\textbf{\textit{v}}(\boldsymbol{\lambda})\cdot\mathrm{\textbf{D}}(\textbf{\textit{u }})\textbf{\textit{n}}d\boldsymbol{\sigma}-\int_{\Gamma} (\textbf{\textit{v}}(\boldsymbol{\lambda})
-\boldsymbol{\lambda})\textbf{\textit{h }}d\boldsymbol{\sigma}=0.
\end{split}
\end{eqnarray}

Due to the fact that $g=0$ on $\Gamma$ and using the Navier boundary condition in~\eqref{problem stokes du noyau}, we have 
\begin{align*}
&2\int_{\Gamma}\textbf{\textit{u}}\cdot\Big(\mathrm{\textbf{D}}(\boldsymbol{\lambda})\textbf{\textit{n}}-\mathrm{\textbf{D}}(\textbf{\textit{v}}(\boldsymbol{\lambda}))\textbf{\textit{n}}\Big)d\boldsymbol{\sigma}
=2\int_{\Gamma}\textbf{\textit{u}}_\tau\cdot\Big(\big[\textbf{D}(\boldsymbol{\lambda})\textbf{\textit{n}}\big]_\tau-\big[\textbf{D}(\textbf{\textit{v}}(\boldsymbol{\lambda}))\textbf{\textit{n}}\big]_\tau\Big) d\boldsymbol{\sigma}\\
&=\alpha\int_{\Gamma}\textit{\textbf{u }}_\tau\cdot\big(\textit{\textbf{v }}(\boldsymbol{\lambda})_\tau-\boldsymbol{\lambda}_\tau\big) d\boldsymbol{\sigma}.
\end{align*}

\noindent Next, using the fact that $\textit{\textbf{v }}(\boldsymbol{\lambda})\cdot\textit{\textbf{n }}=\boldsymbol{\lambda}\cdot\textit{\textbf{n }}$ on $\Gamma$,
\begin{align*}
2\int_{\Gamma}\textbf{\textit{v}}(\boldsymbol{\lambda})\cdot\mathrm{\textbf{D}}(\textbf{\textit{u }})\textbf{\textit{n}}d\boldsymbol{\sigma}
=2\int_{\Gamma}\textbf{\textit{v}}(\boldsymbol{\lambda})_\tau\cdot\big[\mathrm{\textbf{D}}(\textbf{\textit{u }})\textbf{\textit{n}}\big]_\tau d\boldsymbol{\sigma}
+2\int_{\Gamma}(\boldsymbol{\lambda}\cdot\textbf{\textit{n}})[\mathrm{\textbf{D}}(\textbf{\textit{u }})\textbf{\textit{n}}]\textbf{\textit{n}}d\boldsymbol{\sigma}.
\end{align*}

\noindent Finally, using that $\textbf{\textit{h}}\cdot\textbf{\textit{n}}=0$ on $\Gamma$, we obtain 
$$\int_{\Gamma} \big(\textbf{\textit{v}}(\boldsymbol{\lambda})-\boldsymbol{\lambda}\big)\textbf{\textit{h }}d\boldsymbol{\sigma}
=\int_{\Gamma} \big(\textbf{\textit{v}}(\boldsymbol{\lambda})_\tau-\boldsymbol{\lambda}_\tau\big)\textbf{\textit{h }}d\boldsymbol{\sigma}.$$

\noindent Combining these three expressions and after calculation, we obtain 
\begin{align*}\label{calcul sur le bord}
\nonumber &2\int_{\Gamma}\boldsymbol{\lambda}\cdot\mathrm{\textbf{D}}({\textbf{\textit{u}}})\textbf{\textit{n}}d\boldsymbol{\sigma}=2\int_{\Gamma}\textbf{\textit{u}}\cdot\big(\mathrm{\textbf{D}}(\boldsymbol{\lambda})\textbf{\textit{n}}-\mathrm{\textbf{D}}(\textbf{\textit{v}}(\boldsymbol{\lambda}))\textbf{\textit{n}}\big)d\boldsymbol{\sigma}+2\int_{\Gamma}\textbf{\textit{v}}(\boldsymbol{\lambda})\cdot\mathrm{\textbf{D}}(\textbf{\textit{u }})\textbf{\textit{n}}d\boldsymbol{\sigma}\\[4pt]
&
-\int_{\Gamma} (\textbf{\textit{v}}(\boldsymbol{\lambda})
-\boldsymbol{\lambda})\textbf{\textit{h }}d\boldsymbol{\sigma}.
\end{align*}
\noindent Plugging this in~\eqref{calcul.omega.omega'} allows to get~\eqref{condition de compatibilite sur F et e}.

\noindent  Therefore it follows from Theorem \ref{theorem de Stokes dans R^3}, that there exists a solution 
$(\widetilde{\textbf{\textit{z}}},\widetilde{q})\in(\textbf{\textit{W}}^{1,p}_{k}(\R^3)\times {W}^{0,p}_{k}(\R^3))$ satisfying the following Stokes problem:
\begin{eqnarray*}
-\Delta\,\widetilde{\textbf{\textit{z}}}+\nabla\,\widetilde{q}=\textit{\textbf{F}}\quad \text{and} 
\quad \mathrm{div}\,\widetilde{\textbf{\textit{z}}}=e\quad\text{in}\quad\R^3.
\end{eqnarray*}
Using~\eqref{def.F.e}, we obtain
\begin{eqnarray*}
-\Delta\,(\widetilde{\textbf{\textit{z}}}-\widetilde{\textbf{\textit{u}}})+\nabla\,(\widetilde{q}-\widetilde{\pi})=\boldsymbol{0}\quad\text{and} 
\quad \mathrm{div}\,(\widetilde{\textbf{\textit{z}}}-\widetilde{\textbf{\textit{u}}})=0\quad\text{in}\quad\R^3.
\end{eqnarray*}
It follows that $(\widetilde{\textbf{\textit{z}}}-\widetilde{\textbf{\textit{u}}},\widetilde{q}-\widetilde{\pi})$ belongs to 
$\left(W^{1,p}_{k}(\R^3)+{{W}}^{1,2}_{1}(\R^3) \right) \times 
\left(W^{0,p}_{k}(\R^3)+ W^{0,2}_{1}(\R^3) \right)$, then $(\widetilde{\textbf{\textit{z}}}-\widetilde{\textbf{\textit{u}}},\widetilde{q}-\widetilde{\pi})$
 also belongs to $N_{[1-3/p-k]}$. We deduce that there exist $(\boldsymbol{\lambda},\mu)\in N_{[1-3/p-k]}$, then  $\widetilde{\textbf{\textit{z}}}-
 \widetilde{\textbf{\textit{u}}}=\boldsymbol{\lambda}$ and $\widetilde{q}-\widetilde{\pi}=\mu$ which imply that the solution $(\textbf{\textit{u}},\pi)$
 belongs indeed to $W^{1,p}_{k}(\Omega)\times {{W}}^{0,p}_{k}(\Omega)$.\\

\noindent \textsc{Regularity}. Finally, we prove that the solution $(\textit{\textbf{u}},\pi)\in W_k^{1,p}(\Omega)\times W_k^{0,p}(\Omega)$ of ($\mathcal{S}_T$) established previously, belongs to $W_{k+1}^{2,p}(\Omega)\times W_{k+1}^{1,p}(\Omega)$. Here we use regularity arguments on the Stokes problem set in bounded domains and in the whole space~$\R^3$. 
To that end, we introduce the same partition of unity as in Theorem~\ref{theorem Neumann}. Let $(\widetilde{\textbf{\textit{u}}},\widetilde{\pi})\in W_{k}^{1,p}(\R^{3})\times W_k^{0,p}(\R^3)$ be an extension of $(\textit{\textbf{u}},\pi)$
to the whole space $\R^3$. We can write:
$$
\widetilde{\textbf{\textit{u}}}={\varphi}\widetilde{\textbf{\textit{u}}}+
{\psi}\widetilde{\textbf{\textit{u}}}\quad \text{and}\quad \widetilde{\pi}={\varphi}\widetilde{\pi}+{\psi}\widetilde{\pi}.
$$

\noindent Then it is enough to show that the pairs $({\varphi}\widetilde{\textbf{\textit{u}}},{\varphi}\widetilde{\pi})$ and $({\psi}\widetilde{\textbf{\textit{u}}},{\psi}\widetilde{\pi})$ 
belong to $W^{2,p}_{k+1}(\R^3)\times W^{1,p}_{k+1}(\R^3)$. To that end, consider first
\begin{equation}
\label{stokes dans R^3 pour psiu}
 -\Delta\,({\psi}\widetilde{\textbf{\textit{u}}})+\nabla\,({\psi}\widetilde{\pi})=\textbf{\textit{f}}_{1}\quad \mathrm{and}\quad \mathrm{div}\,({\psi}\widetilde{\textbf{\textit{u}}})=\chi_{1} 
 \quad\text{in}\quad\R^3,
 \end{equation}
 where
$$\textbf{\textit{f}}_{1}=\textbf{\textit{f}}\psi
-(2\nabla\widetilde{\textbf{\textit{u}}}\nabla{\psi}+\textbf{\textit{u}}\Delta\,{\psi})+\widetilde{\pi}\nabla\,{\psi}\quad\text{ and } 
\quad \chi_{1}=\chi\psi+\widetilde{\textbf{\textit{u}}}\cdot\nabla\,\psi.
$$
We easily see that $\textbf{\textit{f}}_{1}$ and $\chi_1$ have bounded supports and belong to $L_{loc}^p(\R^3)\times W_{loc}^{1,p}(\R^3)$.
As a consequence, $(\textbf{\textit{f}}_{1},\chi_1)$ belongs to $W^{0,p}_{k+1}(\R^3)\times W^{1,p}_{k+1}(\R^3)$. Using the regularity of the Stokes problem (see Theorem~\ref{theorem de Stokes dans R^3}),  we deduce that the pair $(\psi\widetilde{\textbf{\textit{u}}},\psi\widetilde{\pi})$ 
also belongs to $W^{2,p}_{k+1}(\R^3)\times W^{1,p}_{k+1}(\R^3)$.

%

\noindent Consider now the system
$$
 -\Delta\,({\varphi}\widetilde{\textbf{\textit{u}}})+\nabla\,({\varphi}\widetilde{\pi})
 =\textbf{\textit{f}}_{2}\quad \mathrm{and}\quad \mathrm{div}\,({\varphi}\widetilde{\textbf{\textit{u}}})=\chi_{2},
$$
\noindent where $\textbf{\textit{f}}_2$ and $\chi_2$ have similar expressions as $\textbf{\textit{f}}_1$ and $\chi_1$ 
with $\psi$ remplaced by $\varphi$. It is easy to check that $(\textbf{\textit{f}}_{2},\chi_{2})$ belongs to $L^{\,p}(\Omega_{R+1})\times W^{1,p}(\Omega_{R+1})$. In particular, we have $\widetilde{\textbf{\textit{u}}}={\psi}\widetilde{\textbf{\textit{u}}}$  outside $B_{R+1}$, so the restriction of $\textbf{\textit{u}}$ to $\partial\, B_{R+1}$ belongs to $W^{1+1/p',p}(\partial\, B_{R+1})$. Its clear that $ ({\varphi}\widetilde{\textbf{\textit{u}}},{\varphi}\widetilde{{\pi}})$ belongs to $H^{2}(\Omega_{R+1})\times H^{1}(\Omega_{R+1})$ and satisfies~\eqref{problem mixte}. Then thanks to Theorem~\ref{problem mixte}, we prove that $ ({\varphi}\widetilde{\textbf{\textit{u}}},{\varphi}\widetilde{{\pi}})\in W^{2,p}(\Omega_{R+1})\times W^{1,p}(\Omega_{R+1})$ solution of problem~\eqref{problem mixte}, which also implies 
that $({\varphi}\widetilde{\textbf{\textit{u}}},{\varphi}\widetilde{{\pi}})$ belongs to $W^{2,p}_{k+1}(\R^3)\times W^{1,p}_{k+1}(\R^3)$. Consequently, the pair $(\textit{\textbf{u}},\pi)$ belongs to $W^{2,p}_{k+1}(\Omega)\times W^{1,p}_{k+1}(\Omega)$ if $2\leqslant p\leqslant 6$.
\item[Step 2.] The case $g\neq 0$.\\
\noindent Let $\textbf{\textit{w}}$ be in  $W^{3,p}_{k+1}(\Omega)$ such that $\dfrac{\partial\textbf{\textit{w}}}{\partial\textbf{\textit{n}}}=g$ on $\Gamma$. According to step 1,  the following problem 
\begin{equation}\label{relevement de g et chi}
\begin{cases}
&-\Delta\textbf{\textit{z}}+\nabla\pi=\textbf{\textit{f}}+\Delta(\nabla w)\quad\text{and}\quad\mathrm{div}\,\textit{\textbf{z}}=\chi-\Delta w\quad\text{in}\quad\Omega,\\[4pt]
&\textit{\textbf{z}}\cdot\textit{\textbf{n}}=0\quad\text{and}\quad
2[\mathrm{\textbf{D}}(\textbf{\textit{z}})\textbf{\textit{n}}]_{\tau}+\alpha\textbf{\textit{z}}_{\tau}=\textbf{\textit{K}}\quad\mathrm{on}\quad\Gamma,
\end{cases}
\end{equation}

 has a solution in $W_{k+1}^{2,p}(\Omega)\times W^{1,p}_{k+1}(\Omega)$ if and only if, $\forall (\boldsymbol{\xi},\eta)\in \mathcal{N}^{p'}_{-k+1}(\Omega)$,
\begin{eqnarray}\label{condition de compatibilté pour xi et g}
\displaystyle\int_{\Omega}\big( \textbf{\textit{f}}+\Delta\,(\nabla \textbf{\textit{w}})\big)\cdot\boldsymbol{\xi}d\textbf{\textit{x}}
 -\displaystyle\int_{\Omega}(\chi-\Delta\,\textbf{\textit{w}})\eta d\textbf{\textit{x}}=-
 \int_{\Gamma}\textbf{\textit{K}}\cdot \boldsymbol{\xi}d\boldsymbol{\sigma}.
\end{eqnarray}

\noindent But if $(\boldsymbol{\xi},\eta)$ is in $\mathcal{N}^{p'}_{-k+1}(\Omega)$, then since $\nabla \textbf{\textit{w}}\in W_{k+1}^{2,p}(\Omega)$, 
we can write~\eqref{CC Green's formular p} for the pair $(-\nabla\,\textbf{\textit{w}},0)$ and we obtain
\begin{eqnarray}\label{relation entre g et nabla v}
\nonumber &&\int_{\Omega}\Delta\,(\nabla \textbf{\textit{w}})\cdot\boldsymbol{\xi} d\textbf{\textit{x}}+\displaystyle\int_{\Omega}\Delta \textbf{\textit{w}}\,\,\eta d\textbf{\textit{x}}=-
\int_{\Gamma} g\Big( 2[\mathrm{\textbf{D}}(\boldsymbol{\xi})\textbf{\textit{n}}]\cdot{\textbf{\textit{n}}}-\eta\Big) d\boldsymbol{\sigma}\\
 &&
+\int_{\Gamma} \Big( 2[\textbf{D}(\nabla \textbf{\textit{w}})\textbf{\textit{n}}]_{\tau}+\alpha(\nabla \textbf{\textit{w}})_{\tau}\Big)\cdot \boldsymbol{\xi}d\boldsymbol{\sigma}.
\end{eqnarray}
Combining \eqref{CC TH-strong} and \eqref{relation entre g et nabla v} allows to obtain~\eqref{condition de compatibilté pour xi et g}. 
Thus setting $\textbf{\textit{u}}=\textbf{\textit{z}}+\nabla\,\textbf{\textit{w}}\in W_{k+1}^{2,p}(\Omega)$, then the pair $(\textit{\textbf{u}},\pi)\in W_{k+1}^{2,p}(\Omega)\times W_{k+1}^{1,p}(\Omega)$ is a solution of $(\mathcal{S}_T)$.\\

Now, suppose that $p>6$. The above argument shows that $(\textbf{\textit{u}},\pi)$ belongs to $W^{2,6}_{k+1}(\Omega)\times W^{1,6}_{k+1}(\Omega)$ and we can repeat the same argument with $p=6$ instead of $p=2$ using the fact if $\mathcal{O}$ is a lipschitzian bounded domain, we have $W^{1,6}(\mathcal{O})\hookrightarrow L^{q}(\mathcal{O})$ for any real number $q>1$. We know that the following embedding holds
\begin{center}
$W^{2,p}(\Omega')\hookrightarrow W^{1,q}(\Omega')$,
\end{center}
for any $q\in [1,\infty]$ if $p>3$. Then we have $\textbf{\textit{u}}\in W^{2-1/6,6}(\Gamma)\hookrightarrow W^{1-1/p}(\Gamma)$
 for all $p>6$. Thus establishes the existence of solution $(\textbf{\textit{u}},\pi)$ in $W^{2,p}_{k+1}(\Omega)\times W^{1,p}_{k+1}(\Omega)$ of problem $(\mathcal{S}_T)$ when $p>6$.
\end{proof}

\noindent We finally close this section by the following theorem.
\begin{theo}\label{solution forte f et xi}
Assume that $p\geqslant 2$. Let $\textbf{\textit{f}}\in W^{0,p}_{k+1}(\Omega)$, $\chi\in W^{1,p}_{k+1}(\Omega)$, $g\in W^{1+1/p',p}(\Gamma)$, $\textbf{\textit{h}}\in W^{1-1/p,p}(\Gamma)$ such that  $\textbf{\textit{h}}\cdot\textbf{\textit{n}}=0$ on $\Gamma$ and that the compatibility condition~\eqref{CC TH-strong} is satisfied.
 Then, the Stokes problem $(\mathcal{S}_{T})$ has a solution $(\textbf{\textit{u}},{\pi})\in {W}^{2,p}_{k+1}(\Omega)\times {{W}}^{p}_{k+1}(\Omega)$ unique up to an element of $\mathcal{N}_{k+1}^{p}(\Omega)$. In addition, we have the following estimate:
 
\begin{eqnarray*}
&&\inf _{(\boldsymbol{\lambda},\mu)\in\mathcal{N}_{k+1}^{p}(\Omega)} \left( \Vert \textbf{\textit{u}}+\boldsymbol{\lambda}\Vert_{W^{2,p}_{k+1}(\Omega)}
+\Vert {\pi}+{\mu}\Vert_{{W}^{1,p}_{k+1}(\Omega)}\right) \\ 
&&\leqslant C\Big(\Vert \textbf{\textit{f}}\Vert_{W^{0,p}_{k+1}(\Omega}+\Vert\chi\Vert_{W^{1,p}_{k+1}(\Omega)}+\Vert \textbf{\textit{h}}\Vert_{{W}^{1-1/p,p}(\Gamma)}+\Vert g\Vert_{{W}^{1+1/p',p}(\Gamma)}\Big).
\end{eqnarray*}

\end{theo}

\begin{proof}\quad\\
 Here we prove that problem ($\mathcal{S}_T$) has a solution $(\textit{\textbf{u}},\pi)$ that belongs to $\in W_{k+1}^{2,p}(\Omega)\times W_{k+1}^{1,p}(\Omega)$. To that end, we proceed in two steps.
\item[Step 1.] The case $k<3/p'-1$.\\
First, let extend $\textbf{\textit{f}}$ by zero in $\Omega'$ and denote by $\widetilde{\textbf{\textit{f}}}\in W^{0,p}_{k+1}(\R^3)$ the extended function. Moreover, let $\widetilde{\chi}\in W^{1,p}_{k+1}(\R^3)$ be an extension of $\chi$. We consider the following problem:
\begin{eqnarray}
-\Delta\,\widetilde{\textbf{\textit{w}}}+\nabla\,\widetilde{\eta}=\widetilde{\textbf{\textit{f}}}\quad\text{ and }\quad \mathrm{div}\,\widetilde{\textbf{\textit{w}}}=\widetilde{\chi}\quad\text{ in }\R^3.
\end{eqnarray}
Since $k<3/p'-1$, then $N_{[1-3/p'+k]}=\lbrace(\textbf{\textit{0}},0)\rbrace$ and thus applying Theorem \ref{theorem de Stokes dans R^3} , we deduce that this problem has a solution $(\widetilde{w},\widetilde{\eta})\in W^{2,p}_{k+1}(\R^3)\times W^{1,p}_{k+1}(\R^3)$. Denoting the restriction to $\Omega$ by $(\textbf{\textit{w}},\eta)$ belongs to $W^{2,p}_{k+1}(\Omega)\times W^{1,p}_{k+1}(\Omega)$. Since $\textbf{\textit{w}} \in W^{2,p}_{k+1}(\Omega)$, then we have $[\textbf{D}(\textbf{\textit{w}})\textbf{\textit{n}}]_{\tau}$ belongs to $W^{1-1/p,p}(\Gamma)$. Consider now the following problem:
\begin{eqnarray}\label{relévement f et xi}
\begin{cases}
-\Delta\,\textbf{\textit{v}}+\nabla\,\theta=0,\quad\text{  } \mathrm{div}\,\textbf{\textit{v}}=0\quad\text{ in }\,\,\Omega,\\
\textbf{\textit{v}}\cdot\textbf{\textit{n}}=G\quad \text{ and }
2[\mathrm{\textbf{D}}(\textbf{\textit{v}})\textbf{\textit{n}}]_{\tau}+\alpha\textbf{\textit{v}}_{\tau}
=\textbf{\textit{H}}\quad\text{ on }\Gamma,
\end{cases}
\end{eqnarray}
Where $\textbf{\textit{H}}=-2[\mathrm{\textbf{D}}(\textbf{\textit{w}})\textbf{\textit{n}}]_{\tau}-\alpha\,\textbf{\textit{w}}_{\tau}+\textbf{\textit{h}}$ and $G=g-\textbf{\textit{w}}\cdot\textbf{\textit{n}}$. Its clear that $G\in W^{1+1/p',p}(\Gamma)$ and  $\textbf{\textit{H}}\in W^{1-1/p,p}(\Gamma)$ such that $\textbf{\textit{H}}\cdot\textbf{\textit{n}}=0$ on $\Gamma$. According to Theorem~\ref{solution forte f et xi a support compact}, the problem \eqref{relévement f et xi} has a solution in $W_{k+1}^{2,p}(\Omega)\times W^{1,p}_{k+1}(\Omega)$. Hence, the pair $(\textbf{\textit{u}},\pi)=(\textbf{\textit{w}}+\textbf{\textit{v}},\eta+\theta)$ belongs to $W^{2,p}_{k+1}(\Omega)\times W^{1,p}_{k+1}(\Omega)$ and satisfies problem $(\mathcal{S}_{T})$.\\
\item[Step 2.] The case $k\geqslant 3/p'-1$. We split this step in two cases:
\item[] The case $g=0$. Since $k\geqslant 3/p'-1>0$, then $(\textbf{\textit{f}},\chi)$ belongs to $W^{0,p}_{1}(\Omega)\times W^{1,p}_{1}(\Omega)$. According to Step 1, the problem $(\mathcal{S}_{T})$ has a solution $(\textbf{\textit{u}},\pi)$ belongs to $W^{2,p}_{1}(\Omega)\times W^{1,p}_{1}(\Omega)$. It remains now to prove that $(\textbf{\textit{u}},{\pi})$ belongs to $W^{2,p}_{k+1}(\Omega)\times {W}^{1,p}_{k+1}(\Omega)$.
To that end, we shall use again properties of the Stokes problem in the whole space $\R^3$. Since $\textit{\textbf{u}}\cdot\textit{\textbf{n}}=0$ on $\Gamma$, problem~\eqref{probleme borne omega'} has a solution $(\textit{\textbf{u'}},\pi')\in W^{2,p}(\Omega')\times W^{1,p}(\Omega')$. Set now 
\begin{eqnarray*}\widetilde{\textbf{\textit{u}}}=
\begin{cases}
{\textbf{\textit{u}}}\quad\text{in}\quad\Omega,\\
{\textbf{\textit{u'}}}\quad\text{in}\quad\Omega',\\
\end{cases}
\quad\text{and}\quad\quad
\widetilde{\pi}=
\begin{cases}
\pi\quad\text{in}\quad\Omega,\\
\pi'\quad\text{in}\quad\Omega'.\\
\end{cases}
\end{eqnarray*}

Then clearly, the pair $(\widetilde{\textbf{\textit{u}}},\widetilde{\pi})$ belongs to $ W^{1,p}_1(\R^3)\times {W}^{0,p}_{1}(\R^3)$. In order to apply Theorem~\ref{theorem de Stokes dans R^3} with data $-\Delta\,\widetilde{\textbf{\textit{u}}}+\nabla\,\widetilde{\pi}$ and $\mathrm{div}\,\widetilde{\textbf{\textit{u}}}$, 
we need to show that $(-\Delta\,\widetilde{\textbf{\textit{u}}}+\nabla\,\widetilde{\pi},\mathrm{div}\,\widetilde{\textbf{\textit{u}}})$ belongs to $W^{-1,p}_{k}(\R^3)\times W^{0,p}_{k}(\R^3)$ and satisfies~\eqref{CCR^3}. Therefore, denoting by $\widetilde{\textit{\textbf{f }}}\in W_{k+1}^{0,p}(\R^3)$ the extension of $\textit{\textbf{f }}$ by zero in $\Omega'$. For any $\boldsymbol{\varphi}\in \mathcal{D}(\R^3)$, using the same calculation as in the proof of Theorem\ref{solution forte f et xi a support compact}, we have:
\begin{align}
\label{formule pour le prolongement}
\nonumber \langle -\Delta\,\widetilde{\textbf{\textit{u}}}+\nabla\,\widetilde{\pi},\boldsymbol{\varphi}\rangle_{\boldsymbol{\mathcal{D'}}(\R^3)\times\mathcal{D}(\R^3)}  
&=& \displaystyle\int_{\R^3} \widetilde{\textbf{\textit{f}}}\cdot\boldsymbol{\varphi}d\textbf{\textit{x}}+2\int_{\Gamma}\boldsymbol{\varphi}\cdot\Big(\mathrm{\textbf{D}}({\textbf{\textit{u}}})\textbf{\textit{n}}
-\mathrm{\textbf{D}}({\textbf{\textit{u'}}})\textbf{\textit{n}}\Big)d\boldsymbol{\sigma}\\
 &&+\int_{\Gamma}\big(\boldsymbol{\varphi}\cdot\textbf{\textit{n}}\big)\big(\pi'-\pi\big)d\boldsymbol{\sigma}
- \int_{\Gamma}\boldsymbol{\varphi}\cdot\textbf{\textit{n}},\chi d\boldsymbol{\sigma}.
\end{align}
Since $\mathcal{D}(\R^3)$ is dense in $W^{1,p'}_{-k}(\R^3)$ and $\widetilde{\textbf{\textit{f}}}$ belongs to $W^{0,p}_{k+1}(\R^3)$, 
then \eqref{formule pour le prolongement} is still valid for any
$\boldsymbol{\varphi} \in W^{1,p'}_{-k}(\R^3)$ which implies that $-\Delta\,\widetilde{\textbf{\textit{u}}}+\nabla\,\widetilde{\pi}$ belongs to $W_{k}^{-1,p}(\R^3)$.\\
\noindent Now for any $\phi \in \mathcal{D}(\R^3)$, we have
\begin{align}\label{formule pour le prolongement2}
 \langle\mathrm{div}\,\widetilde{\textbf{\textit{u}}},\phi\rangle_{\mathcal{D'}(\R^3)\times\mathcal{D}(\R^3)} 
=\displaystyle\int_{\Omega}\chi\,\phi d\textbf{\textit{x}}.
\end{align}
Since $\chi$ belongs to $W^{1,p}_{k+1}(\Omega)$ and due to the density of $\mathcal{D}(\R^3)$ in $W^{0,p'}_{-k}(\R^3)$, then~\eqref{formule pour le prolongement2} is still valid for any $\phi \in W^{0,p'}_{-k}(\R^3)$, which implies that  $\mathrm{div}\,\widetilde{\textbf{\textit{u}}}$ belongs to $W^{0,p}_{k}(\R^3)$.\\
\noindent As a consequence, we have proved that $(-\Delta\,\widetilde{\textbf{\textit{u}}}+\nabla\,\widetilde{\pi},\mathrm{div}\,\widetilde{\textbf{\textit{u}}})$ belongs to $W_k^{-1,p}(\R^3)\times W_k^{0,p}(\R^3)$. Using the same calculation as in the proof of Theorem~\ref{solution forte f et xi a support compact}, we prove that $(-\Delta\,\widetilde{\textbf{\textit{u}}}+\nabla\,\widetilde{\pi},\mathrm{div}\,\widetilde{\textbf{\textit{u}}})$ satisfy~\eqref{CCR^3}. Therefore it follows from Theorem \ref{theorem de Stokes dans R^3}, that there exists a solution 
$(\widetilde{\textbf{\textit{z}}},\widetilde{q})\in(\textbf{\textit{W}}^{1,p}_{k}(\R^3)\times {W}^{0,p}_{k}(\R^3))$ satisfying the following Stokes problem:
\begin{eqnarray*}
-\Delta\,(\widetilde{\textbf{\textit{z}}}-\widetilde{\textbf{\textit{u}}})+\nabla\,(\widetilde{q}-\widetilde{\pi})=\boldsymbol{0}\quad\text{and} 
\quad \mathrm{div}\,(\widetilde{\textbf{\textit{z}}}-\widetilde{\textbf{\textit{u}}})=0\quad\text{in}\quad\R^3.
\end{eqnarray*}
It follows that $(\widetilde{\textbf{\textit{z}}}-\widetilde{\textbf{\textit{u}}},\widetilde{q}-\widetilde{\pi})$ belongs to 
$\left(W^{1,p}_{k}(\R^3)+{{W}}^{1,p}_{1}(\R^3) \right) \times 
\left(W^{0,p}_{k}(\R^3)+ W^{0,p}_{1}(\R^3) \right)$, then $(\widetilde{\textbf{\textit{z}}}-\widetilde{\textbf{\textit{u}}},\widetilde{q}-\widetilde{\pi})$
 also belongs to $N_{[1-3/p-k]}$. We deduce that there exist $(\boldsymbol{\lambda},\mu)\in N_{[1-3/p-k]}$, such that $\widetilde{\textbf{\textit{z}}}-
 \widetilde{\textbf{\textit{u}}}=\boldsymbol{\lambda}$ and $\widetilde{q}-\widetilde{\pi}=\mu$ which imply that the solution $(\textbf{\textit{u}},\pi)$
 belongs indeed to $W^{1,p}_{k}(\Omega)\times {{W}}^{0,p}_{k}(\Omega)$.\\

\noindent Finally, to prove that the solution $(\textit{\textbf{u}},\pi)\in W_k^{1,p}(\Omega)\times W_k^{0,p}(\Omega)$ of ($\mathcal{S}_T$) established previously,
actually belongs to $W_{k+1}^{2,p}(\Omega)\times W_{k+1}^{1,p}(\Omega)$, we can proceed as in the proof of Theorem~\ref{solution forte f et xi a support compact} with the use of the partition 
of unity~\eqref{partition de l'unite}.
\item[] For the case $g\neq 0$, we proceed as in the same way as in the proof of Theorem~\ref{solution forte f et xi a support compact}, we prove that the pair $(\textit{\textbf{u}},\pi)$ belongs to $W_{k+1}^{2,p}(\Omega)\times W_{k+1}^{1,p}(\Omega)$ is a solution of $(\mathcal{S}_T)$.\\


\end{proof}

\end{document}